\documentclass[12pt,leqno]{amsart}
\usepackage{amsmath}
\usepackage{amssymb}
\def\overset#1#2{{\mathrel{\mathop {{#2}_{}}\limits^{#1}}}}
\def\underset#1#2{{\mathrel{\mathop {{}_{} {#2}}\limits_{{#1}_{}}}}}
\def\upplim_#1{\underset{#1}{\overline\lim}\;}
\def\lowlim_#1{\underset{#1}{\underline\lim}\;}
\newtheorem{corollary}[equation]{Corollary}
\newtheorem{definition}[equation]{\indent{\it Definition}\rm }

\newtheorem{lemma}[equation]{Lemma}
\newtheorem{proposition}[equation]{Proposition}

\newtheorem{theorem}[equation]{Theorem}

\newcommand{\C}{{\mathbb{C}}}

\renewcommand{\P}{{\mathbb{P}}}
\newcommand{\B}{{\mathbb{B}}}
\newcommand{\Q}{{\mathbb{Q}}}

\newcommand{\R}{{\mathbb{R}}}

\newcommand{\ric}{\mathrm{Ric}}
\newcommand{\supp}{\mathrm{Supp}\,}

\newcommand{\Z}{\mathbb{Z}}

\numberwithin{equation}{section}

\title[Degeneracy second main theorem and Schmidt's subspace theorem]{Generalizations of degeneracy second main theorem and Schmidt's subspace theorem} 

\author{Si Duc Quang}

\begin{document}

\begin{abstract}
In this paper, by introducing the notion of ``\textit{distributive constant}'' of a family of hypersurfaces with respect to a projective variety, we prove a second main theorem in Nevanlinna theory for meromorphic mappings with arbitrary families of hypersurfaces in projective varieties. Our second main theorem generalizes and improves previous results for meromorphic mappings with hypersurfaces, in particular for algebraically degenerate mappings and for the families of hypersurfaces in subgeneral position. The analogous results for the holomorphic curves with finite growth index from a complex disc into a project variety, and for meromorphic mappings on a complete K\"{a}hler manifold are also given. For the last aim, we will prove a Schmidt's subspace theorem for an arbitrary families of homogeneous polynomials, which is the counterpart in Number theory of our second main theorem. Our these results are generalizations and improvements of many previous results.
\end{abstract}

\def\thefootnote{\empty}

\maketitle

\section{Introduction and Results}

Let $V$ be a subvariety of $\P^N(\C)$ of dimension $n>0$. Let $l\geq n$ and $q\geq l+1.$ Hypersurfaces $Q_1,\ldots,Q_q$ in $\mathbb P^N(\mathbb C)$ are said to be in $l$-subgeneral position with respect to $V$ if 
$$Q_{j_0}\cap\cdots\cap Q_{j_{l}}\cap V=\emptyset\text{ for every }1\leq j_0<\cdots<j_{l}\leq q.$$
If  $l=n$ then we say that $Q_1,\ldots,Q_q$ are in \textit{general position} w.r.t $V$. If $V=\P^N(\C)$, we just say that $Q_1,\ldots,Q_q$ are in general (resp. $N$-subgeneral) position in $\P^N(\C)$.

\vskip0.2cm 
In 1933, H. Cartan \cite{Ca} established the most important second main theorem for meromorphic mappings from $\C^m$ into $\P^N(\C)$ and hyperplanes as follows (here we use the standard notations from Nevanlinna theory which can be found in Section 2).

\vskip0.2cm
\noindent
\textbf{Theorem A.} {\it Let $f:  {\C}^m \to \P^N(\C)$ be a linearly nondegenerate meromorphic mapping and $\{H_i\}_{i=1}^q$ be $q$ hyperplanes in general position in $\P^N(\C)$. Then we have}
$$\| \ (q-(N+1))T_f(r) \leq \sum_{i=1}^q N^{[N]}(r,f^*H_i)+ o(T_f(r)).$$
Here, by the notation ``$\| P$'' we mean the assertion $P$ holds for all $r\in [0,\infty)$ excluding a Borel subset $E$ of the interval $[0,\infty)$ with $\int_E dr<\infty$. 
%If $ \lim\limits_{r\rightarrow R_0}\sup\dfrac{T(r,r_0)}{\log\frac{R_0}{R_0-r}}= \infty,$ then 
The Nevanlinna's defect of $f$ with respect to a hypersurface $Q$ of degree $d$ truncated to level $M$ is defined by
$$ \delta^{[M]}_{f,*}(Q)=1-\lim\mathrm{sup}\dfrac{N^{[M]}(r,f^*Q)}{dT_f(r)}.$$
Then, from Theorem A we obtain the estimate of total defect for $f$ with family $\{H_i\}_{i=1}^q$ as follows:
$$ \sum_{i=1}^q \delta^{[N]}_{f,*}(H_i)\le n+1.$$

In order to deal with the case of arbitrary meromorphic mapping (not require the linear non-degenerate condition), we have to consider $f$ as a linearly nondegenerate mapping into the smallest projective subspace $V (\approx\P^n(\C))$ of $\P^N(\C)$ containing $f(\C^m)$ and consider $\{H_i\}_{i=1}^q$ as a family of $q$ hyperplanes in $N$-subgeneral position in $V$. By introducing the notion of Nochka weight, in 1983 Nochka \cite{Noc83} gave the following second main theorem for this case.

\vskip0.2cm
\noindent
\textbf{Theorem B} (cf. \cite{Noc83,No05}). {\it Let $f:  {\C}^m \to \P^n(\C)$ be a meromorphic mapping and $\{H_i\}_{i=1}^q$ be hyperplanes in $N-$subgeneral position in $\P^n(\C)$ $(N\ge n)$. Then we have
$$\|\  (q-(2N-n+1))T_f(r) \leq \sum_{i=1}^q N^{[n]}(r,f^*H_i)+ o(T_f(r)).$$}
Hence, the bound above of the total defect obtained from this theorem is $(2N-n+1).$

In 1992, A. E. Eremenko and M. L. Sodin \cite{ES} firstly proved a second main theorem for arbitrary holomorphic curve from $\C$ into $\P^N(\C)$ intersecting $q$ hypersurfaces $\{Q_i\}_{i=1}^q$ in general position of the form
$$\|\  (q-2N)T_f(r) \leq \sum_{i=1}^q \dfrac{1}{\deg Q_i}N(r,f^*Q_i)+ o(T_f(r)).$$
Then, the total defect obtained from this second main theorem is bounded above by $2N$.
 
We note that, in 1979, B. Shiffman \cite{Sh79} conjectured that if $f$ is an algebraically nondegenerate holomorphiccurve from $\C$ into $\P^N(\C)$ then the total defect of $f$ with $q$ hypersurfaces $\{Q_i\}_{i=1}^q$ in general position is bounded above by $N+1$. This conjectured is proved by M. Ru \cite{Ru04} in 2004. The result of M. Ru is improved by T. T. H. An and H. T. Phuong \cite{AP} with an explicit truncation level for counting functions. Later on, in 2009, M. Ru \cite{Ru09} proved a more general second main theorem as follows.

\vskip0.2cm
\noindent
\textbf{Theorem C} (cf. \cite{Ru09}) {\it Let $V$ be a smooth  projective subvariety of dimension $n\ge 1$ of $\P^N(\C)$. Let $Q_1,\ldots,Q_q$ be $q$ hypersurfaces in $\P^N(\C)$ in general position w.r.t $V$. Let $f:\C\to V$ be an algebraically nondegenerate holomorphic map. Then, for every $\epsilon >0$,
$$\bigl\|\ (q-n-1-\epsilon) T_f(r)\le\sum_{i=1}^q\frac{1}{\deg Q_i}N(r,f^*Q_i)+o(T_f(r)).$$}
We see that, the bound above of the total defect obtained from this theorem is $n+1$, and it is believed to be sharp.

Note that all mentioned above results, the mappings always are assumed to be algebraically nondegenerate. In order to consider the case of algebraically degenerate mappings, we need to consider such mappings as algebraically nondegenerate mappings into the smallest subvariety $V$ of $\P^N(\C)$ containing their images with a family of hypersurface in$\P^N(\C)$ in $N$-subgeneral position w.r.t $V$. 

In some years later, by developing the method of M. Ru, many authors have given second main theorems for algebraically nondegenerate meromorphic mappings into $V$ of dimension $n$ with hypersurfaces in $N$-subgeneral position. We list here the name of some authors together their bound above of the total defect:
\begin{itemize}
\item 2012, Z. Chen, M. Ru and Q. Yan \cite{CRY}; $ \sum_{i=1}^q \delta^{[M]}_{f,*}(Q_i)\le l(n+1)$ for some $M$.
\item 2014, L. Shi and M. Ru \cite{SR}; $\sum_{i=1}^q \delta^{[M]}_{f,*}(Q_i)\le \dfrac{l(l-1)(n+1)}{l+n-2}$ for some $M$.
\item 2016, L. Giang \cite{G}; $\sum_{i=1}^q \delta^{[M_0]}_{f,*}(Q_i)\le l(n+1)$ with an explicit estimate for the truncation level $M_0$.
\end{itemize}
However, these results cannot deduce Theorem C of M. Ru or the classical second main theorem of H. Cartan. The difficulty comes from the fact that there is no Nochka weights for the families of hypersurfaces. 

Recently, in \cite{Q19} the author proposed the replacing hypersurface technique, which allows to generalize the second main theorem for families of hypersurfaces in general position to the case of families of hypersurfaces in subgeneral posiition without using Nochka's weight. He proved the following.

\vskip0.2cm
\noindent
\textbf{Theorem D} (cf. \cite{Q19}) {\it Let $V$ be a smooth projective subvariety of dimension $n\ge 1$ of $\P^N(\C)$. Let $Q_1,\ldots,Q_q$ be hypersurfaces in $\P^N(\C)$ in $l-$subgeneral position with respect to $V$, $\deg Q_i=d_i\ (1\le i\le q)$. Let $d$ be the least common multiple of $d_1,\ldots,d_q$, i.e., $d=l.c.m.(d_1,\ldots,d_q)$. Let $f:\C^m\to V$ be an algebraically nondegenerate meromorphic mapping. Then, for every $\epsilon >0$, 
$$\bigl \|\ \left (q-(l-n+1)(n+1)-\epsilon\right) T_f(r)\le\sum_{i=1}^q\frac{1}{d_i}N^{[M_0]}_{Q_i(f)}(r)+o(T_f(r)),$$
where $M_0=\left[\deg (V)^{n+1}e^nd^{n^2+n}(l-n+1)^n(2n+4)^n(q!)^n\epsilon^{-n}\right]$.}

Here, the notation $[x]$ stands for the biggest integer not exceed the real number $x$.

This technique has been applied by later authors for the analogous problems such as: second main theorem for moving hypersurfaces in subgeneral position, non-integrated defect relation of meromorphic mappings on K\"{a}hler manifolds intersecting hypersurfaces in subgeneral position, subspace theorem in number theory for family of homogeneous polynomials in subgeneral position. In particular, in 2019 Q. Ji, Q. Yan and G. Yu \cite{JYY} applied this method for the case of families of hypersurfaces in $l$-subgeneral with index $\kappa$ and obtained: 
\begin{itemize}
\item $\sum_{i=1}^q \delta_{f,*}(Q_i)\le\left(\dfrac{l-n}{\max\{1,\min\{l-n,\kappa\}\}}+1\right)(n+1)$, without trucation level.
\end{itemize}
Here, a family of hypersurfaces $\{Q_i\}_{i=1}^q$ is said to be in $l$-general position with index $\kappa$ in $V$ if it is in $l$-subgeneral position w.r.t $V$ and $\dim\bigcap_{j=0}^{s}\Q_{i_j}\cap V\le n-s-1$ for all $0\le s\le \kappa-1$.

Actually, in \cite{JYY}, Q. Ji, Q. Yan and G. Yu also obtained another bound above of the total defect (see \cite[Theorem 1.1]{JYY}) as follows
$$ \sum_{i=1}^q \delta_{f,*}(Q_i)\le\max\left\{\dfrac{l}{2\kappa};1\right\}(n+1).$$
Unfortunately, the proof of result is not correct. The gap comes from that they used a lemma of Levin \cite[Lemma 10.1]{L} to choose a common basis $\mathcal B=\{s_1,\ldots,s_{H_{\phi}(X)}\}$ of $\hat{V}$ with respect to two filtrations corresponding with two families of $n$ hypersufaces in general position, say $\{D_{1,z},\ldots,D_{\kappa,z},H'_{\kappa+1,z},$ $ ...,  H'_{n,z}\}$ and $\{D_{\kappa+1,z},\ldots,D_{2\kappa,z},H''_{\kappa+1,z}, ... , H''_{n,z}\}$. However, from the proof of \cite[Lemma 10.1]{L}, for an element $s_j\in W_{\bf i}\setminus W_{\bf i'}$, this $s_j$ may not be of the form $[r_1^{i_1},\ldots,r_2^{i_n}r]$, where ${\bf i}=(i_1,\ldots,i_n)$ (since, the condition ${\bf i'}=(i'_1,\ldots,i_n')>{\bf i}=(i_1,\ldots,i_n)$ in the lexicographic order does not implies $i'_s\ge i_s$ for all $s$). Hence, the proof of M. Ru for (3.6) in \cite{Ru04} can not be applied for this case, then the estimate (3.7) in \cite{JYY} does not hold for both filtrations. 
% We may see the contradiction from the inequality (3.9) in \cite{JYY}. We reformulate it for the simple case where $X=\P^n(\C)$, $s_i\in\C[x_0,\ldots,x_n]_N$, $D_{i,z}$ are hypersurfaces of degree $d$, as follows:
% $$ \sum_{i=1}^{H_{\P^n(\C)}(N)}\mathrm{div}(s_i)\ge \left (\dfrac{N^{n+1}}{(n+1)!}+O(N^n)\right)\cdot\dfrac{1}{d}\cdot \left (D_{1,z}+\cdots +D_{2\kappa,z}\right).$$
% Then, the degree of the divisor on the left hand side is just $\binom{N+n}{n}\cdot N=\dfrac{N^{n+1}}{n!}+O(N^n)$, but that of the right hand side is $\left (\dfrac{2\kappa N^{n+1}}{(n+1)!}+O(N^n)\right)>\dfrac{N^{n+1}}{n!}+O(N^n)$ if $2\kappa >n+1$. Actually, if their method was correct, the sum $(D_{1,z}+\cdots +D_{2\kappa,z})$ can be replaced by 
% $$(D_{1,z}+\cdots +D_{2\kappa,z}+H'_{\kappa+1,z}+\cdots H'_{n,z}+H''_{\kappa+1,z}+\cdots H''_{n,z})$$  
% with suitable hypersurfaces $H'_{j,z},H''_{j,z}$ of degree $d$, and hence the degree of the right hand side becomes to $\left(\dfrac{2n N^{n+1}}{(n+1)!}+O(N^n)\right)$.
 In our opinion, from their proof, the correction of \cite[Theorem 1.1]{JYY} should be as follows.

\vskip0.2cm
\noindent
\textbf{Theorem E.} {\it Let $X$ be a complex projective variety with $\dim X = n$, and let $D_1,\ldots,D_q$ be effective Cartier divisors in $m$-subgeneral position with index $\kappa (>1)$ on $X$. Suppose that there exists an ample divisor $A$ on $X$ and positive integers $d_j$ such that $D_j\sim d_jA$ for $j=1,\ldots,q$. Let $f:\C\rightarrow X$ be an algebraically nondegenerate holomorphic curve. Then, for every $\epsilon >0$,
$$ \bigl\|\ \dfrac{1}{d_j}m_f(r,Q_j)\le \left (\dfrac{l}{\kappa}(n+1)+\epsilon\right)T_f(r,A).$$}

Our purpose in this paper is to study the more general case. We will consider the case of arbitrary families of hypersurfaces, not required to be in subgeneral position. To do so, we will introduce a notion of \textit{distributive constant} $\Delta$ of a family of hypersurface $\{Q_i\}_{i=1}^q$ of $\P^N(\C)$ in a subvariety $V\subset\P^N(\C)$ (see Section 3 for detail), where $V\not\subset\mathrm{supp}Q_i\ (i=1,\ldots, q)$, as follows:
$$\Delta:=\underset{\Gamma\subset\{1,\ldots,q\}}\max\dfrac{\sharp\Gamma}{\dim V-\dim\left (\bigcap_{j\in\Gamma}\mathrm{supp} Q_j\right )\cap V}.$$
We will develop the replacing hypersurfaces technique in \cite{Q19} to the case of arbitrary families of hypersurfaces (see Lemma \ref{3.2}, which is one of two most important key lemmas in this paper). Another most important one is Lemma \ref{3.1}, in which we prove that the product of powers of $n$ positive number is bounded above by the power $\Delta$ of the product of these numbers.

The first main theorem of this paper is stated as follows.
 \begin{theorem}\label{1.1} 
Let $V\subset\P^N(\C)$ be a smooth complex projective variety of dimension $n\ge 1$. Let $\{Q_1,\ldots,Q_q\}$ be a family of hypersurfaces in $\P^N(\C)$ with the distributive constant $\Delta$ in $V$, $\deg Q_i=d_i\ (1\le i\le q)$, and let $d$ be the least common multiple of $d_1,\ldots,d_q$. Let $f:\C^m\to V$ be an algebraically nondegenerate meromorphic mapping. Then, for every $\epsilon >0$, 
$$\bigl \|\ \left (q-\Delta (n+1)-\epsilon\right) T_f(r)\le\sum_{i=1}^q\frac{1}{d_i}N^{[M_0]}(r,f^*Q_i)+o(T_f(r)),$$
where $M_0=\left[d^{n^2+n}\deg (V)^{n+1}e^n\Delta^n(2n+4)^n(n+1)^n(q!)^n\epsilon^{-n}\right]$.
\end{theorem}
Hence, this theorem implies the estimate of the total defect as follows:
\begin{itemize}
\item $\sum_{i=1}^q \delta^{[M_0]}_{f,*}(Q_i)\le \Delta(n+1)$, with an explicit estimate for $M_0$.
\end{itemize}
Also, as indicated in Section 3, we have the following remarks:
\begin{itemize}
\item If $\{Q_i\}_{i=1}^q$ is in general position then $\Delta =1$.
\item If $\{Q_i\}_{i=1}^q$ is in $l$-subgeneral position then $\Delta\le l-n+1$.
\item If $\{Q_i\}_{i=1}^q$ is in $l$-general position with index $\kappa$ then $\Delta\le\dfrac{l-n+\kappa}{\kappa}$.
\end{itemize}
Therefore, Theorem \ref{1.1} generalizes all above mentioned results, also improves the result of Q. Ji, Q. Yan and G. Yu.

The above theorem immediately gives us the following corollary. 
\begin{corollary}
Let $V\subset\P^N(\C)$ be a smooth complex projective variety of dimension $n\ge 1$. Let $\{Q_1,\ldots,Q_q\}$ be a family of hypersurfaces in $\P^N(\C)$ with the distributive constant $\Delta$ in $V$. If $q>\Delta (n+1)$ then every meromorphic mapping from $\C^m$ into $V\setminus\bigcup_{i=1}^q\supp Q_i $ is algebraically degenerate.
\end{corollary}

For the case of one dimension, i.e., $f$ is a holomorphic curve from the disc $\B^1(R_0)=\{z\in\C; |z|< R_0\}$ into $\P^N(\C)$, according to M. Ru-N. Sibony \cite{RS}, the growth index of $f$ is defined by
$$ c_{f}=\mathrm{inf}\left\{c>0\ \biggl |\int_{0}^{R_0}\mathrm{exp}(cT_{f}(r,r_0))dr=+\infty\right\}.$$
For convenient, we will set $c_f=+\infty$ if $\left\{c>0\ \bigl |\int_{0}^{R_0}\mathrm{exp}(cT_{f}(r))dr=+\infty\right\}=\varnothing$. 
In the connection to Theorem \ref{1.1}, we will prove an analogous result for such curves with finite growth index as follows.

\begin{theorem}\label{1.2} 
Let $V\subset\P^N(\C)$ be a smooth complex projective variety of dimension $n\ge 1$. Let $\{Q_1,\ldots,Q_q\}$ be a family of hypersurfaces in $\P^N(\C)$ with the distributive constant $\Delta$ in $V$, $\deg Q_i=d_i\ (1\le i\le q)$, and let $d$ be the least common multiple of $d_1,\ldots,d_q$. Let $f:\B^1(R_0)\to V\ (0<R_0<+\infty)$ be an algebraically nondegenerate meromorphic mapping with $c_f<+\infty$. Then, for every $\epsilon >0, 0<r_0<R_0$, there exists a positive number $\epsilon'$ such that
\begin{align*}
\left(q-\Delta(n+1)-\epsilon\right)T_f(r,r_0)&\le \sum_{i=1}^q\frac{1}{d}N^{[M_0]}(r,r_0,f^*Q_i)+\dfrac{M_0c_fT_f(r,r_0)}{2d(2n+1)(n+1)(q!)\delta}\\
&+O(\log T_f(r,r_0))
\end{align*}
for all $r\in (r_0,R_0)$ outside a set $E$ with $\int_E\mathrm{exp}((c_f+\epsilon')T_f(r,r_0))dr<+\infty$, where 
$$M_0=\left[d^{n^2+n}\deg (V)^{n+1}e^n\Delta^n(2n+4)^n(n+1)^n(q!)^n\epsilon^{-n}\right].$$
\end{theorem}

For a more general case, where the determining domain of the mappings is a complete K\"{a}hler manifold, in 1985, H. Fujimoto \cite{Fu} defined the non-integrated defect of $f$ with respect to a hypersurface $Q$ truncated to level $\mu_0$ by
$$\delta_f^{[\mu_0]}:= 1- \inf\{\eta\ge 0: \eta \text{ satisfies condition }(*)\}.$$
Here, the condition (*) means there exists a bounded non-negative continuous function $h$ on $M$ whose order of each zero is not less than $\min \{f^*Q, \mu_0\}$ such that
$$d\eta\Omega_f +\dfrac{\sqrt{-1}}{2\pi}\partial\bar\partial\log h^2\ge [\min\{f^*Q, \mu_0\}].$$
We will prove the version of the non-integrated defect relation of Theorem \ref{1.1} as follows.
\begin{theorem}\label{1.3} 
Let $M$ be an $m$-dimensional complete K\"ahler manifold  and $\omega$ be a K\"ahler form of $M.$  Assume that the universal covering of $M$ is biholomorphic to a ball in $\mathbb C^m.$ Let $V\subset\P^N(\C)$ be a smooth complex projective variety of dimension $n\ge 1$. Let $\{Q_1,\ldots,Q_q\}$ be a family of hypersurfaces in $\P^N(\C)$ with the distributive constant $\Delta$ in $V$, $\deg Q_i=d_i\ (1\le i\le q)$, and let $d$ be the least common multiple of $d_1,\ldots,d_q$. Let $f:M\rightarrow V$ be an algebraically nondegenerate meromorphic mapping. Assume that for some $\rho \ge 0,$ there exists a bounded continuous function $h \geq 0$ on $M$ such that
$$\rho\Omega_f + dd^c \log h^2 \geq \ric \ \omega.$$
Then, for each  $\epsilon>0,$ 
$$\sum_{j=1}^q {\delta}^{[M_0]}_{f} (Q_j) \le \Delta(n+1)+ \epsilon +\dfrac{\rho\epsilon M_0}{\delta(2n+1)(n+1)(q!)d},$$
where $M_0=\left[d^{n^2+n}\deg (V)^{n+1}e^n\Delta^n(2n+4)^n(n+1)^n(q!)^n\epsilon^{-n}\right].$
\end{theorem}
We see that this result also generalizes the results in \cite{TT},\cite{Y13},\cite{RSo},\cite{Q17} and \cite{Q21}.

In the last section of this paper, we will prove the counterpart in number theory of our Theorem \ref{1.1} according to Vojta's dictionary, which gives the corresponding between Nevanlinna theory and Diophantine approximation (see \cite{V}). We will study the degenerate case of the Schmidt subspace theorem for arbitrary families of hypersurfaces. By the notations in Section 5, our last main result is stated as follows.

\begin{theorem}\label{1.5} 
Let $k$ be a number field, $S$ be a finite set of places of $k$ and let $V$ be an irreducible projective subvariety of $\P^N$ of dimension $n$ defined over $k$. Let $\{Q_1,\ldots, Q_q\}$ be a family of $q$ homogeneous polynomials of $\bar k[x_0,\ldots,x_N]$ with the distributive constant $\Delta$ in $V(\bar k)$. Then for each $\epsilon >0$, 
$$\sum_{v\in S}\sum_{j=1}^q\dfrac{\lambda_{Q_j,v}({\bf x})}{\deg Q_j}\le (\Delta(n+1)+\epsilon)h({\bf x})$$
for all ${\bf x}\in\P^N(k)$ outside a union of closed proper subvarieties of $V$.
\end{theorem}
This result covers all previous results on subspace theorem with families of hypersurfaces (see \cite{CZ},\cite{EF2},\cite{CRY},\cite{L} and \cite{Q19b} for the reference).

\section{Notation and Auxialiary results}

{\bf 2.1. Counting function.}\ We use the following usual notations:
\begin{align*}
 \|z\|& := \big(|z_1|^2 + \dots + |z_m|^2\big)^{1/2} \text{ for }z = (z_1,\dots,z_m) \in \mathbb C^m,\\
\B^m(r) &:= \{ z \in \mathbb C^m : \|z\| < r\},\\
S(r) &:= \{ z \in \mathbb C^m : \|z\| = r\}\ (0<r<\infty),\\
v_{m-1}(z) &:= \big(dd^c \|z\|^2\big)^{m-1},\\
\sigma_m(z)&:= d^c\log\|z\|^2 \land \big(dd^c\log\|z\|^2\big)^{m-1} \text{on} \quad \mathbb C^m \setminus \{0\}.
\end{align*}
 For a divisor $\nu$ on  a ball $\B^m(R_0)$ of $\mathbb C^m$, with $R_0>0$ and a positive integer $M$ or $M= \infty$, we define the counting function of $\nu$ by
\begin{align*}
&\nu^{[M]}(z)=\min\ \{M,\nu(z)\},\\
&n(t) =
\begin{cases}
\int\limits_{|\nu|\,\cap \B^m(t)}
\nu(z) v_{m-1} & \text  { if } m \geq 2,\\
\sum\limits_{|z|\leq t} \nu (z) & \text { if }  m=1. 
\end{cases}
\end{align*}
Similarly, we define $n^{[M]}(t).$

Define
$$ N(r,r_0,\nu)=\int\limits_{r_0}^r \dfrac {n(t)}{t^{2m-1}}dt \quad (0<r_0<r<R_0).$$

Similarly, define  $N(r,r_0,\nu^{[M]})$ and denote it by $N^{[M]}(r,r_0,\nu)$.

For a meromorphic function $\varphi$ on $\B^m(R_0)$, denote by $\nu_\varphi$ its divisor of zeros and set
$$N_{\varphi}(r,r_0)=N(r,r_0,\nu_{\varphi}), \ N_{\varphi}^{[M]}(r,r_0)=N^{[M]}(r,r_0,\nu_{\varphi})\ (r_0<r<R_0).$$

For brevity, we will omit the character $^{[M]}$ if $M=\infty$.

\vskip0.2cm
\noindent
{\bf 2.2. Characteristic function and first main theorem.}\ Fix a homogeneous coordinates system $(x_0 : \dots : x_N)$ on $\mathbb P^N(\mathbb C)$. Let $f : \mathbb \B^m(R_0) \longrightarrow \mathbb P^N(\mathbb C)$ be a meromorphic mapping with a reduced representation $\tilde f = (f_0 , \ldots , f_N)$. Set $\Vert \tilde f \Vert = \big(|f_0|^2 + \dots + |f_N|^2\big)^{1/2}$.
The characteristic function of $f$ is defined by 
$$ T_f(r,r_0):=\int_{r_0}^r\dfrac{dt}{t^{2m-1}}\int\limits_{\B^m(t)}f^*\Omega\wedge v^{m-1}, \ (0<r_0<r<R_0),$$
where $\Omega$ is the Fubini-Study form on $\P^N(\C)$. By Jensen's formula, we will have
\begin{align*}
T_f(r,r_0)= \int\limits_{S(r)} \log\Vert \tilde f \Vert \sigma_m -
\int\limits_{S(r_0)}\log\Vert \tilde f\Vert \sigma_m +O(1), \text{ (as $r\rightarrow R_0$)}.
\end{align*}

In this paper, we call a hypersurfaces in $\P^N(\C)$ a nonzero homogeneous polynomial in $\C[x_0,x_1,\ldots,x_N]$. Let $Q$ be a hypersurface in $\P^N(\C)$ of degree $d$, which is of the form
$$ Q({\bf x})=\sum_{I\in\mathcal T_d}a_I{\bf x}^I, $$
where $\mathcal T_d=\{(i_0,\ldots,i_N)\in\mathbb Z_+^{N+1}\ ;\ i_0+\cdots +i_N=d\}$, ${\bf x} =(x_0,\ldots,x_N)$, ${\bf x}^I=x_0^{i_0}\cdots x_N^{i_N}$ with $I=(i_0,\ldots,i_N)\in\mathcal T_d$ and $a_I\ (I\in\mathcal T_d)$ are constants, not all zeros. 
If there is no confusion, we also denote again by $Q$ the divisor generated by the hypersurface $Q$. We will denote by $Q^{*}$ the support of $Q$, i.e., 
$$Q^*=\{(x_0:\cdots:x_N)\in\P^N(\C)\ |\ Q(x_0,\ldots,x_N)=0\}.$$
In the case $d=1$, we call $Q$ a hyperplane of $\P^N(\C)$.

The proximity function of $f$ with respect to $Q$, denoted by $m_f (r,r_0,Q)$, is defined by
$$m_f (r,r_0,Q)=\int_{S(r)}\log\dfrac{\|\tilde f\|^d}{|Q(\tilde f)|}\sigma_m-\int_{S(r_0)}\log\dfrac{\|\tilde f\|^d}{|Q(\tilde f)|}\sigma_m,$$
where $Q(\tilde f)=Q(f_0,\ldots,f_N)$. This definition is independent of the choice of the reduced representation of $f$. 

We denote by $f^*Q$ the pullback of the divisor $Q$ by $f$. Then, $f^*Q$ identifies with the divisor of zeros $\nu^0_{Q(\tilde f)}$ of the function $Q(\tilde f)$. By Jensen's formula, we have
$$N(r,r_0,f^*Q)=N_{Q(\tilde f)}(r,r_0)=\int_{S(r)}\log |Q(\tilde f)|\sigma_m-\int_{S(r_0)}\log |Q(\tilde f)| \sigma_m.$$
The first main theorem in Nevanlinna theory for meromorphic mappings and hypersurfaces is stated as follows.

\begin{theorem}[First Main Theorem] Let $f : \mathbb B^m(R_0) \to \P^n(\C)$ be a holomorphic map, and let $Q$ be a hypersurface in $\P^n(\C)$ of degree $d$. If $f (\mathbb B(\R_0)) \not \subset Q^*$, then for every real number $r$ with $r_0 < r < R_0$,
$$dT_f (r,r_0)=m_f (r,r_0,Q) + N(r,r_0,f^*Q)+O(1),$$
where $O(1)$ is a constant independent of $r$.
\end{theorem}

\vskip0.2cm
\noindent
\textbf{2.3. Some results on general Wronskian.}
For each meromorphic function $g$ on $\B^m(R_0)$ and  an $m$-tuples $\alpha =(\alpha_1,\ldots,\alpha_m)\in\Z^m_+$, we set $|\alpha|=\sum_{i=1}^m\alpha_i$ and define 
$${\mathcal D}^{\alpha}g=\frac{\partial^{|\alpha|}g}{\partial^{\alpha_{1}}z_1\ldots \partial^{\alpha_{m}}z_m}.$$

Repeating the argument in \cite[Proposition 4.5]{Fu}, we have the following proposition.
\begin{proposition}[{see \cite[Proposition 4.5]{Fu}}]\label{2.2}
Let $\Phi_0,\ldots,\Phi_{N}$ be meromorphic functions on $\B^m(R_0)$ such that $\{\Phi_0,\ldots,\Phi_{N}\}$ 
are  linearly independent over $\C.$
Then  there exists an admissible set $\{\alpha_i=(\alpha_{i1},\ldots,\alpha_{im})\}_{i=0}^{N} \subset \Z^m_+$ with $|\alpha_i|=\sum_{j=1}^{m}\alpha_{ij}\le i \ (0\le i \le N)$ satisfying the following two properties:

(i)\  $\{{\mathcal D}^{\alpha_i}\Phi_0,\ldots,{\mathcal D}^{\alpha_i}\Phi_{N}\}_{i=0}^{N}$ is linearly independent over $\mathcal M$ (the field of all meromorphic functions on $\B^m(R_0)$),\ i.e., $\det{({\mathcal D}^{\alpha_i}\Phi_j)}_{0\le i,j\le N}\not\equiv 0,$

(ii) $\det \bigl({\mathcal D}^{\alpha_i}(h\Phi_j)\bigl)=h^{N+1}\det \bigl({\mathcal D}^{\alpha_i}\Phi_j\bigl)$ for
any meromorphic function $h$ on $\B^m(R_0).$

\end{proposition}
We note that $\alpha_0,\ldots,\alpha_{N}$ are chosen uniquely in an explicit way (see \cite{Fu}). Then we define the general Wronskian of the mapping $\Phi =(\Phi_0,\ldots,\Phi_{N})$ by 
$$ W(\Phi):=\det{({\mathcal D}^{\alpha_i}\Phi_j)}_{0\le i,j\le N}. $$

In \cite{RS}, M. Ru and S. Sogome gave the following lemma on logarithmic derivative.

\begin{proposition}[{see \cite{RS}, Proposition 3.3}]\label{pro2.2}
 Let $L_0,\ldots ,L_{N}$ be linear forms of $N+1$ variables and assume that they are linearly independent. Let $F$ be a meromorphic mapping of the ball $\B^m(R_0)\subset\C^m$ into $\P^{N}(\C)$ with a reduced representation $\tilde F=(F_0,\ldots ,F_{N})$ and let $(\alpha_0,\ldots ,\alpha_N)$ be an admissible set of $F$. Set $l=|\alpha_0|+\cdots +|\alpha_N|$ and take $t,p$ with $0< tl< p<1$. Then, for $0 < r_0 < R_0,$ there exists a positive constant $K$ such that for $r_0 < r < R < R_0$,
$$\int_{S(r)}\biggl |z^{\alpha_0+\cdots +\alpha_N}\dfrac{W_{\alpha_0,\ldots ,\alpha_N}(F_0,\ldots ,F_{N})}{L_0(\tilde F)\cdots L_{N}(\tilde F)}\biggl |^t\sigma_m\le K\biggl (\dfrac{R^{2m-1}}{R-r}T_F(R,r_0)\biggl )^p.$$
\end{proposition}
Here $z^{\alpha_i}=z_1^{\alpha_{i1}}\cdots z_m^{\alpha_{im}},$ where $\alpha_i=(\alpha_{i1},\ldots,\alpha_{im})\in\mathbb Z^m_+$ and  
$$W_{\alpha_0,\ldots ,\alpha_N}(F_0,\ldots ,F_{N})=\det (\mathcal D^{\alpha_i}F_j)_{0\le i,j\le N}.$$

The following theorem is due to M. Ru and N. Sibony \cite{RS}
\begin{theorem}[see \cite{RS}]\label{2.4new}
Let $f:\B^1(R_0)\rightarrow\P^N(\C)\ (0<R_0<+\infty)$ be a linearly nondegenerate holomorphic curve with $c_f<+\infty$, and let $\tilde f=(f_0,\ldots,f_N)$ be a reduced representation of $f$. Let $H_1,\ldots,H_q$ be arbitrary hyperplanes in $\P^N(\C)$. Denote by $W(\tilde f)$ the general Wronskian of $(f_0,\ldots,f_N)$.
Then, for every $\epsilon >0$,
\begin{align*}
\int\limits_{S(r)}\max_{K}&\log\prod_{i\in K}\dfrac{\|\tilde f\|}{|H_i(\tilde f)|}\sigma_m+ N_{W(\tilde f)}(r)\le (N +1)T_f(r,r_0)\\
&+\dfrac{N(N+1)}{2}(1+\epsilon)(c_f+\epsilon)T_f(r,r_0)+O(\log T_f(r,r_0)).
\end{align*}
for all $r\in (r_0,R_0)$ outside a subset $E$ with $\int_E\mathrm{exp}((c_f+\epsilon)T_f(r,r_0))dr<+\infty$, where the maximum is taken over all subsets $K$ of $\{1,\ldots,q\}$ so that $\{H_i | i\in K\}$ is linearly independent.
\end{theorem}

\vskip0.2cm 
\noindent
{\bf 2.4. Meromorphic mappings on $\C^m$.} If $f$ is a meromorphic mapping of $\C^m$ into $\P^N(\C)$ (i.e., $R_0=+\infty$), we will write $T_f(r),N_{\varphi}(r),m_f(r,Q)$ for $T_f(r,1),N_{\varphi}(r,1),m_f(r,1,Q)$ respectively.

Let $\varphi$ be a nonzero meromorphic function on $\mathbb B^m(R_0)$, which is occasionally regarded as a meromorphic map into $\mathbb P^1(\mathbb C)$. The proximity function of $\varphi$ is defined by
$$m(r,\varphi)=\int_{S(r)}\log \max\ (|\varphi|,1)\sigma_m.$$
The Nevanlinna's characteristic function of $\varphi$ is define as follows
$$ T(r,\varphi)=N_{\frac{1}{\varphi}}(r)+m(r,\varphi). $$
Then 
$$T_\varphi (r)=T(r,\varphi)+O(1)\ \text{ as } r\to +\infty.$$
The function $\varphi$ is said to be small (with respect to $f$) if $\|\ T_\varphi (r)=o(T_f(r))$.

The lemma on logarithmic derivative in this case is stated as follows.
\begin{lemma}[{see \cite{NO}}]
Let $f$ be a nonzero meromorphic function on $\C.$ Then we have
$$\biggl\|\quad m\biggl(r,\dfrac{\mathcal D^{\alpha}(\tilde f)}{f}\biggl)=O(\log^+T(r,f))\ (\alpha\in \mathbb Z^m_+).$$
\end{lemma}
 
The following general form of the second main theorem is due to M. Ru \cite{Ru97}
\begin{theorem}[see \cite{Ru97}]\label{2.4}
Let $\C^m\rightarrow\P^N(\C)$ be a linearly nondegenerate holomorphic curve with a reduced representation $\tilde f=(f_0,\ldots,f_N)$. Let $H_1,\ldots,H_q$ be arbitrary hyperplanes in $\P^N(\C)$. Denote by $W(\tilde f)$ the general Wronskian of $(f_0,\ldots,f_N)$.
Then, for every $\epsilon >0$,
$$\int\limits_{S(r)}\max_{K}\log\prod_{i\in K}\dfrac{\|\tilde f\|}{|H_i(\tilde f)|}\sigma_m+ N_{W(\tilde f)}(r)\le (N +1+\epsilon)T_f(r),$$
where the maximum is taken over all subsets $K$ of $\{1,\ldots,q\}$ so that $\{H_i| i\in K\}$ is linearly independent.
\end{theorem}
We note that, M. Ru proved the above theorem for the case $m=1$, but this theorem also holds for the general case. The proof of this theorem for the general case is the same with that of the special case $m=1$.

\vskip0.2cm
\noindent
{\bf 2.5. Chow weights and Hilbert weights.} We recall the notion of Chow weights and Hilbert weights from \cite{Ru09}.

Let $X\subset\P^N$ be a projective variety of dimension $n$ and degree $\delta$ defined over a number field $k$. The Chow form of $X$ is the unique polynomial, up to a constant scalar, 
$$F_X(\textbf{u}_0,\ldots,\textbf{u}_n) = F_X(u_{00},\ldots,u_{0N};\ldots; u_{n0},\ldots,u_{nN})$$
in $N+1$ blocks of variables $\textbf{u}_i=(u_{i0},\ldots,u_{iN}), i = 0,\ldots,n$ with the following
properties: $F_X$ is irreducible in $k[u_{00},\ldots,u_{nN}]$; $F_X$ is homogeneous of degree $\delta$ in each block $\textbf{u}_i, i=0,\ldots,n$; and $F_X(\textbf{u}_0,\ldots,\textbf{u}_n) = 0$ if and only if $X\cap H_{\textbf{u}_0}\cap\cdots\cap H_{\textbf{u}_n}\ne\varnothing$, where $H_{\textbf{u}_i}, i = 0,\ldots,n$, are the hyperplanes given by
$$u_{i0}x_0+\cdots+ u_{iN}x_N=0.$$

Let ${\bf c}=(c_0,\ldots, c_N)$ be a tuple of real numbers and $t$ be an auxiliary variable. We consider the decomposition
\begin{align*}
F_X(t^{c_0}u_{00},&\ldots,t^{c_N}u_{0N};\ldots ; t^{c_0}u_{n0},\ldots,t^{c_N}u_{nN})\\ 
& = t^{e_0}G_0(\textbf{u}_0,\ldots,\textbf{u}_N)+\cdots +t^{e_r}G_r(\textbf{u}_0,\ldots, \textbf{u}_N).
\end{align*}
with $G_0,\ldots,G_r\in k[u_{00},\ldots,u_{0N};\ldots; u_{n0},\ldots,u_{nN}]$ and $e_0>e_1>\cdots>e_r$. The Chow weight of $X$ with respect to ${\bf c}$ is defined by
\begin{align*}
e_X({\bf c}):=e_0.
\end{align*}
For each subset $J = \{j_0,\ldots,j_n\}$ of $\{0,\ldots,N\}$ with $j_0<j_1<\cdots<j_n,$ we define the bracket
\begin{align*}
[J] = [J]({\bf u}_0,\ldots,{\bf u}_n):= \det (u_{ij_t}), i,t=0,\ldots,n,
\end{align*}
where $\textbf{u}_i = (u_{i_0},\ldots,u_{iN})\ (1\le i\le n)$ denote the blocks of $N+1$ variables. Let $J_1,\ldots,J_\beta$ with $\beta=\binom{N+1}{n+1}$ be all subsets of $\{0,\ldots,N\}$ of cardinality $n+1$.

Then $F_X$ can be written as a homogeneous polynomial of degree $\delta$ in $[J_1],\ldots,[J_\beta]$. We may see that for $\textbf{c}=(c_0,\ldots,c_N)\in\R^{N+1}$ and for any $J$ among $J_1,\ldots,J_\beta$,
\begin{align*}
\begin{split}
[J](t^{c_0}u_{00},\ldots,t^{c_N}u_{0N},&\ldots,t^{c_0}u_{n0},\ldots,t^{c_N}u_{nN})\\
&=t\sum_{j\in J}c_j[J](u_{00},\ldots,u_{0N},\ldots,u_{n0},\ldots,u_{nN}).
\end{split}
\end{align*}

For $\textbf{a} = (a_0,\ldots,a_N)\in\mathbb Z^{N+1}$ we write ${\bf x}^{\bf a}$ for the monomial $x^{a_0}_0\cdots x^{a_N}_N$. Denote by $k[x_0,\ldots,x_N]_u$ the vector space of homogeneous polynomials in $k[x_0,\ldots,x_N]$ of degree $u$ (including $0$). For an ideal $I$ in $k[x_0,\ldots,x_N]$, we put $I_u :=k[x_0,\ldots,x_N]_u\cap I$. Let $I(X)$ be the prime ideal in $k[x_0,\ldots,x_N]$ defining $X$.  The Hilbert function $H_X$ of $X$ is defined by, for $u = 1, 2,\ldots,$
\begin{align*}
H_X(u):=\dim (k[x_0,\ldots,x_N]_u/I(X)_u).
\end{align*}
By the usual theory of Hilbert polynomials,
\begin{align*}
H_X(u)=\delta\cdot\frac{u^N}{N!}+O(u^{N-1}).
\end{align*}
The $u$-th Hilbert weight $S_X(u,{\bf c})$ of $X$ with respect to the tuple ${\bf c}=(c_0,\ldots,c_N)\in\mathbb R^{N+1}$ is defined by
\begin{align*}
S_X(u,{\bf c}):=\max\left (\sum_{i=1}^{H_X(u)}{\bf a}_i\cdot{\bf c}\right),
\end{align*}
where the maximum is taken over all sets of monomials ${\bf x}^{{\bf a}_1},\ldots,{\bf x}^{{\bf a}_{H_X(u)}}$ whose residue classes modulo $I$ form a basis of $k[x_0,\ldots,x_N]_u/I_u.$

The following theorems are due to J. Evertse and R. Ferretti.
\begin{theorem}[{Theorem 4.1 \cite{EF1}}]\label{2.12}
Let $X\subset\P^N$ be an algebraic variety of dimension $n$ and degree $\delta$ defined over a number field $k$. Let $u>\delta$ be an integer and let ${\bf c}=(c_0,\ldots,c_N)\in\mathbb R^{N+1}_{\geqslant 0}$.
Then
$$ \frac{1}{uH_X(u)}S_X(u,{\bf c})\ge\frac{1}{(n+1)\delta}e_X({\bf c})-\frac{(2n+1)\delta}{u}\cdot\left (\max_{i=0,\ldots,N}c_i\right). $$
\end{theorem}

\begin{lemma}[{see \cite[Lemma 5.1]{EF2}, also \cite{Ru09}}]\label{2.13}
Let $Y\subset\P^N$ be an algebraic variety of dimension $n$ and degree $\delta$ defined over a number field $k$. Let ${\bf c}=(c_1,\ldots, c_q)$ be a tuple of positive reals. Let $\{i_0,\ldots,i_n\}$ be a subset of $\{1,\ldots,q\}$ such that
$$Y \cap \{y_{i_0}=\cdots =y_{i_n}=0\}=\varnothing.$$
Then
$$e_Y({\bf c})\ge (c_{i_0}+\cdots +c_{i_n})\delta.$$
\end{lemma}

\section{Distributive constant and lemma on replacing hypersurfaces}

We prove the following two key lemmas of this papers.

\begin{lemma}\label{3.1}
Let $t_0,t_1,\ldots,t_n$ be $n+1$ integers such that $1=t_0<t_1<\cdots <t_n$, and let $\Delta =\underset{1\le s\le n}\max\dfrac{t_s-t_0}{s}$. 
Then for every $n$ real numbers $a_0,a_1,\ldots,a_{n-1}$  with $a_0\ge a_1\ge\cdots\ge a_{n-1}\ge 1$, we have
$$ a_0^{t_1-t_0}a_1^{t_2-t_1}\cdots a_{n-1}^{t_{n}-t_{n-1}}\le (a_0a_1\cdots a_{n-1})^{\Delta}.$$
\end{lemma}

\begin{proof}
Let $s$ be an index, $1\le s\le n$ such that $\Delta=\dfrac{t_s-t_0}{s}$. 

If there exists an index $k (1\le k\le s-1)$ such that $t_s-t_{s-k}< k\Delta$ then $t_{s-k}-t_0>(s-k)\Delta$, and hence $\dfrac{t_{s-k}-t_0}{s-k}>\Delta$, which contradicts the maximality of $\Delta$. Therefore, $t_s-t_{s-k}\ge k\Delta$ for all $1\le k\le s-1$.

Also, if there exists an index $k (1\le k\le n-s)$ such that $t_{s+k}-t_s>k\Delta$ then $t_{s+k}-t_0>(s+k)\Delta$, and hence $\dfrac{t_{s+k}-t_0}{s+k}>\Delta$, which also contradicts the maximality of $\Delta$. Therefore, we also have $t_{s+k}-t_s\le k\Delta$ for all $1\le k\le n-s$.

We set $m_n=\Delta$ and define
\begin{align*}
m_{n-1}&=t_{n}-t_{n-1}+\max\{0,m_n-\Delta\}, \\
m_{n-2}&=t_{n-1}-t_{n-2}+\max\{0,m_{n-1}-\Delta\},\\
\ldots&\ldots\\
m_{0}&=t_{1}-t_{0}+\max\{0,m_1-\Delta\}.
\end{align*}

Denote by $u$ the smallest index such that $s\le u\le n$ and $m_u\le \Delta$. Suppose that $u>s$ then $m_j>\Delta\ (\forall s\le j\le u-1)$, and hence we have the following estimate:
\begin{align*}
m_{u-1}&=t_u-t_{u-1}>\Delta,\\ 
m_{u-2}&=t_{u-1}-t_{u-2}+m_{u-1}-\Delta=t_u-t_{u-2}-\Delta>\Delta,\\
m_{u-3}&=t_{u-2}-t_{u-3}+m_{u-2}-\Delta=t_u-t_{u-3}-2\Delta>\Delta,\\ 
\ldots&\ldots\\
m_{s}&=t_{s+1}-t_{s}+m_{s+1}-\Delta=t_u-t_s-(u-s-1)\Delta>\Delta.
\end{align*}
This implies that
$$ t_u-t_s>(u-s)\Delta. $$
This is a contradiction. Therefore, the supposition is false. Then $u=s$, and hence $m_s\le \Delta$.

We easily have the following estimate:
\begin{align*}
m_{s-1}&=t_{s}-t_{s-1}\ge\Delta,\\
m_{s-2}&=t_{s-1}-t_{s-2}+m_{s-1}-\Delta=t_s-t_{s-2}-\Delta\ge 2\Delta-\Delta=\Delta,\\
m_{s-3}&=t_{s-2}-t_{s-3}+m_{s-2}-\Delta=t_s-t_{s-3}-2\Delta\ge 3\Delta-2\Delta=\Delta,\\
\ldots&\ldots\\
m_{1}&=t_{2}-t_{1}+m_{2}-\Delta=t_s-t_{1}-(s-2)\Delta\ge (s-1)\Delta-(s-2)\Delta=\Delta,\\
m_{0}&=t_{1}-t_{0}+m_{1}-\Delta=t_s-t_{0}-(s-1)\Delta= s\Delta-(s-1)\Delta=\Delta.
\end{align*}
On the other hand, for each $i\in\{0,\ldots,n-2\}$ we have
$$\dfrac{a_i^{m_i}a_{i+1}^{\Delta}}{a_i^{t_{i+1}-t_i}a_{i+1}^{m_{i+1}}}=\dfrac{a_{i}^{\max\{0,m_{i+1}-\Delta\}}}{a_{i+1}^{m_{i+1}-\Delta}}\ge\left(\dfrac{a_i}{a_{i+1}}\right)^{\max\{0,m_{i+1}-\Delta\}}\ge 1.$$
This implies that
$$a_i^{t_{i+1}-t_i}a_{i+1}^{m_{i+1}}\le a_i^{m_i}a_{i+1}^{\Delta},$$
for every $0\le i\le n-2$. Then, we easily have that
\begin{align*}
a_0^{t_1-t_0}a_1^{t_2-t_1}\cdots a_{n-1}^{t_{n}-t_{n-1}}&= a_0^{t_1-t_0}a_1^{t_2-t_1}\cdots a_{n-2}^{t_{n-1}-t_{n-2}}a_{n-1}^{m_{n-1}}\\ 
&\le a_0^{t_1-t_0}a_1^{t_2-t_1}\cdots a_{n-2}^{t_{n-2}-t_{n-3}}a_{n-2}^{m_{n-2}}a_{n-1}^{\Delta}\\
&\le a_0^{t_1-t_0}a_1^{t_2-t_1}\cdots a_{n-2}^{t_{n-3}-t_{n-4}}a_{n-3}^{m_{n-3}}a_{n-2}^{\Delta}a_{n-1}^{\Delta}\\
&\ldots\\
&\le a_0^{t_1-t_0}a_1^{m_1}a_2^\Delta\cdots a_{n-1}^{\Delta}\\
&\le a_0^{m_0}a_1^{\Delta}a_2^\Delta\cdots a_{n-1}^{\Delta}\\
&=(a_0a_1\cdots a_{n-1})^{\Delta}.
\end{align*}
We complete the proof of the lemma.
\end{proof}

Let $k$ be a number field, and $\bar k$ be its algebraic closure. Let $V$ be a projective subvariety of $\P^N$ of dimension $n$ defined over $k$. In this paper, we fix a homogeneous coordinates system $(x_0:\cdots:x_N)$ of $\P^N$. 

We recall that a hypersurface $Q$ in $\P^N(k)$ is a nonzero homogeneous polynomial in $\bar{k}[x_0,\ldots,x_N]$ and $Q^*$ denotes its support (in $V(\bar k)$).

\begin{lemma}\label{3.2}
Let $k$ be a number field and let $V$ be a projective subvariety of $\P^N(k)$ of dimension $n$. Let $Q_0,\ldots,Q_{l}$ be $l$ hypersurfaces in $\P^N(k)$ of the same degree $d\ge 1$, such that $\bigcap_{i=0}^{l}Q_i^*\cap V(\bar k)=\varnothing$ and
$$\dim\left (\bigcap_{i=0}^{s}Q_i^*\right )\cap V(\bar k)=n-u\ \forall t_{u-1}\le s<t_u,1\le u\le n,$$
where $t_0,t_1,\ldots,t_n$ integers with $0=t_0<t_1<\cdots<t_n=l$. Then there exist $n+1$ hypersurfaces $P_0,\ldots,P_n$ in $\P^N(k)$ of the forms
$$P_u=\sum_{j=0}^{t_{u}}c_{uj}Q_j, \ c_{uj},\ u=0,\ldots,n,$$
such that $\left (\bigcap_{u=0}^{n}P_u^*\right )\cap V(\bar k)=\varnothing.$
\end{lemma}
\begin{proof} Set $P_0=Q_0$. We will construct $P_1,\ldots,P_n$ as follows.

Step 1. Firstly, we will construct $P_1$. For each irreducible component $\Gamma$ of dimension $n-1$ of $P_0^*\cap V(\bar k)$, we put 
$$V_{1\Gamma}=\{c=(c_0,\ldots,c_{t_1})\in k^{t_1+1}\ ;\ \Gamma\subset Q_c^*,\text{ where }Q_c=\sum_{j=0}^{t_1}c_jQ_j\}.$$
Here, $Q_c$ may be zero polynomial and its support $Q^*_c$ is $\P^n(\bar k)$. We see that $V_{1\Gamma}$ is a subspace of $k^{t_1+1}$. Since $\dim \left(\bigcap_{j=0}^{t_1}Q_j^*\right)\cap V(\bar k)\le n-2$, there exists $i \ (1\le i\le t_1+1)$ such that $\Gamma\not\subset Q_i^*$. Then $V_{1\Gamma}$ is a proper subspace of $k^{t_1+1}$. Since the set of irreducible components of dimension $n-1$ of $P_0^*\cap V(\bar k)$ is at most countable, 
$$ k^{t_1+1}\setminus\bigcup_{\Gamma}V_{1\Gamma}\ne\varnothing. $$
Hence, there exists $(c_{10},c_{11},\ldots,c_{1t_1})\in k^{t_1+1}$ such that
$$ \Gamma\not\subset P_1^*$$
for all irreducible components $\Gamma$ of dimension $n-1$ of $P_0^*\cap V(\bar k)$, where
$P_1=\sum_{j=0}^{t_1}c_{1j}Q_j.$
This implies that $\dim \left(P_0^*\cap P_1^*\right)\cap V(\bar k)\le n-2.$

Step 2. For each irreducible component $\Gamma'$ of dimension $n-2$ of $\left(P_0^*\cap P_1^*\right)\cap V(\bar k)$, put 
$$V_{2\Gamma'}=\{c=(c_0,\ldots,c_{t_2})\in k^{t_2+1}\ ;\ \Gamma\subset Q_c^*,\text{ where }Q_c=\sum_{j=0}^{t_2}c_jQ_j\}.$$
Hence, $V_{2\Gamma'}$ is a subspace of $k^{t_2+1}$. Since $\dim \left(\bigcap_{i=0}^{t_2}Q_i^*\right)\cap V\le n-3$, there exists $i, (0\le i\le t_2)$ such that $\Gamma'\not\subset Q_i^*$. Then, $V_{2\Gamma'}$ is a proper subspace of $k^{t_2+1}$. Since the set of irreducible components of dimension $n-2$ of $\left(P_0^*\cap P_1^*\right)\cap V(\bar k)$ is at most countable, 
$$ k^{t_2+1}\setminus\bigcup_{\Gamma'}V_{2\Gamma'}\ne\varnothing. $$
Therefore, there exists $(c_{20},c_{21},\ldots,c_{2t_2})\in k^{t_2+1}$ such that
$$ \Gamma'\not\subset P_3^*$$
for all irreducible components of dimension $n-2$ of $P_0^*\cap P_1^*\cap V(\bar k)$, where
$P_3=\sum_{j=0}^{t_2}c_{2j}Q_j.$
It implies that $\dim \left(P_1^*\cap P_2^*\cap P_3^*\right)\cap V(\bar k)\le n-3.$

Repeating again the above steps, after the $n^{\rm th}$-step we get hypersurfaces $P_0,\ldots,P_n$ satisfying
$$ \dim\left(\bigcap_{j=0}^tP_j^*\right)\cap V(\bar k)\le n-t-1\ (0\le t\le n). $$
In particular, $\left(\bigcap_{j=0}^{n}P_j^*\right)\cap V(\bar k)=\varnothing.$ We complete the proof of the lemma.
\end{proof}

\begin{definition}\label{3.3}
Let $V$ be a projective subvariety of $\P^N(k)$ of dimension $n$, and let $Q_1,\ldots,Q_q$ be $q$ hypersurfaces in $\P^N(k)$. We define the distributive constant of the family $\{Q_1,\ldots,Q_q\}$ in $V$ by
$$ \Delta:=\underset{\Gamma\subset\{1,\ldots,q\}}\max\dfrac{\sharp\Gamma}{n-\dim\left (\bigcap_{j\in\Gamma} Q_j^*\right )\cap V(\bar k)}.$$
\end{definition}
Here, we note that $\dim\varnothing =-\infty$.

Remark:

$\bullet$ If $Q_1,\ldots,Q_q\ (q\ge n+1)$ are in general position with respect to $V$ then 
$$ \Delta=\max\left\{\dfrac{1}{n-(n-1)},\dfrac{2}{n-(n-2)},\ldots,\dfrac{n}{n-(n-n)}\right\}=1.$$

$\bullet$ If $Q_1,\ldots,Q_q\ (q\ge m+1)$ are in $m-$subgeneral position with respect to $V$ then we may see that for every subset $\{Q_{i_1},\ldots,Q_{i_k}\}\ (1\le k\le m)$, one has
$$ \dim\left(\bigcap_{j=1}^kQ_{i_j}\right)\le \min\{n-1,m-k\},$$
and hence
\begin{align*}
\Delta&\le\max\left\{\dfrac{1}{n-(n-1)},\dfrac{2}{n-(n-1)},\ldots,\dfrac{m-n+1}{n-(n-1)},\dfrac{m-n+2}{n-(n-2)},\ldots,\dfrac{m}{n-(m-m)}\right\}\\
&=m-n+1.
\end{align*}
For more general, we have the following definition.

\begin{definition}\label{de3.4}
Let $k$ be a number field and let $V$ be a smooth projective subvariety of $\P^N(k)$ of dimension $n$. Let $Q_0,\ldots,Q_{l}$ be $l$ hypersurfaces in $\P^N(k)$. We say that the family $\{Q_0,\ldots,Q_{l}\}$ is in $(t_1,t_2,\ldots,t_n)$-subgeneral position with respect to $V$ if for every $1\le s\le n$ and $t_s+1$ hypersurfaces $Q_{j_0},\ldots,Q_{j_{t_s}}$, we have
$$ \dim\bigcap_{i=0}^{t_s}Q_{j_i}^*\cap V(\bar k) \le n-s-1.$$
\end{definition}

Remark
\begin{itemize}
\item[(a)] If $\{Q_0,\ldots,Q_{l}\}$ is in $(t_1,t_2,\ldots,t_n)$-subgeneral position with respect to $V$ then its distributive constant in $V$ satisfying
$$ \Delta\le\max_{1\le k\le n}\dfrac{t_k}{n-(n-k)}=\max_{1\le k\le n}\dfrac{t_k}{k}.$$   
\item[(b)] If $Q_1,\ldots,Q_q\ (q\ge m+1)$ are in $m-$subgeneral position with respect to $V$ with index $\kappa$,  one has
$$ \dim\left(\bigcap_{j=1}^kQ_{i_j}\right)\le n-\kappa-(k-(m-n+\kappa-1))=m-k-1.$$
Then $\{Q_0,\ldots,Q_{l}\}$ is in $(1,2,\ldots,\kappa-1,m-n+\kappa,m-n+\kappa+1,\ldots,m-1,m)$-subgeneral position with respect to $V$
and hence
\begin{align*}
\Delta&\le\max\biggl\{\dfrac{1}{n-(n-1)},\dfrac{2}{n-(n-2)},\ldots,\dfrac{\kappa-1}{n-(n-\kappa+1)},\\
&\quad\quad\quad\quad\quad\quad\quad\quad\quad\ldots,\dfrac{(m-n)+\kappa}{n-(n-\kappa)},\dfrac{(m-n)+\kappa+1}{n-(n-\kappa+1)},\ldots,\frac{m}{n}\biggl\}\\
&=\max\left\{1,\dfrac{m-n}{\kappa}+1\right\}=\dfrac{m-n+\kappa}{\kappa}.
\end{align*}
\end{itemize}

\section{General second main theorems and non-integrated defect  relation}

\begin{proof}[{\bf Proof of Theorem \ref{1.1}}]

It is suffice for us to consider the case where $\Delta<\dfrac{q}{n+1}$. Note that $\Delta\ge 1$, and hence $q>n+1$. If there exists $i\in\{1,\ldots,q\}$ such that $\bigcap_{\underset{j\ne i}{j=1}}^qQ_j^*\cap V\ne\varnothing$ then 
$$ \Delta\ge\dfrac{q-1}{n}>\dfrac{q}{n+1}.$$
This is a contradiction. Therefore, $\bigcap_{\underset{j\ne i}{j=1}}^qQ_j^*\cap V=\varnothing$ for all $i\in\{1,2,\ldots,q\}$.

Since the number of hypersurfaces occurring in this proof is finite, we may choose a positive constant $c$ such that for each given hypersurface $Q$ in this proof, we have
$$ Q({\bf x})\le c\|{\bf x}\|^{\deg Q} $$
for all ${\bf x}=(x_0,\ldots,x_n)\in\C^{n+1}$. 

Firstly, we will prove the theorem for the case where all hypersurfaces $Q_i\ (1\le i\le q)$ are of the same degree $d$. 
We denote by $\mathcal I$ the set of all permutations of the set $\{1,\ldots,q\}$. Denote by $n_0$ the cardinality of $\mathcal I$, $n_0=q!$, and we write
$\mathcal I=\{I_1,\ldots,I_{n_0}\}$,
where $I_i=(I_i(0),\ldots,I_i(q-1))\in\mathbb N^q$ and $I_1<I_2<\cdots <I_{n_0}$ in the lexicographic order.

For each $I_i\in\mathcal I$, since $\bigcap_{j=1}^{q-1}Q_{I_i(j)}^*\cap V=\varnothing$, there exist $n+1$ integers $t_{i,0},t_{i,1},\ldots,t_{i,n}$ with $0=t_{i,0}<\cdots<t_{i,n}=l_i$, where $l_i\le q-2$ such that $\bigcap_{j=0}^{l_i}Q_{I_i(j)}^*\cap V=\varnothing$ and
$$\dim\left (\bigcap_{j=0}^{s}Q_{I_i(j)}^*\right )\cap V=n-u\ \forall t_{i,u-1}\le s<t_{i,u},1\le u\le n.$$
Then, $\Delta >\dfrac{t_{i,u}-t_{i,0}}{u}$ for all $1\le u\le n.$ Denote by $P_{i,0},\ldots,P_{i,n}$ the hypersurfaces obtained in Lemma \ref{3.1} with respect to the hypersurfaces $Q_{I_i(0)},\ldots,Q_{I_i(l_i)}$. We may choose a positive constant $B\ge 1$, commonly for all $I_i\in\mathcal I$, such that
$$ |P_{i,j}({\bf x})|\le B\max_{0\le s\le t_{i,j}}|Q_{I_i(j)}({\bf x})|, $$
for all $0\le j\le n$ and for all ${\bf x}=(x_0,\ldots,x_n)\in\C^{n+1}$. 

Consider a reduced representation $\tilde f=(f_0,\ldots,f_n): \C^m\rightarrow \C^{N+1}$ of $f$. Fix an element $I_i\in\mathcal I$. Denote by $S(i)$ the set of all points $z\in \C\setminus\bigcup_{i=1}^qQ_i(\tilde f)^{-1}(\{0\})$ such that
$$ |Q_{I_i(0)}(\tilde f)(z)|\le |Q_{I_i(1)}(\tilde f)(z)|\le\cdots\le |Q_{I_i(q-1)}(\tilde f)(z)|.$$
Since $\bigcap_{j=0}^{l_i}Q_{I_i(j)}^*\cap V=\varnothing$, by Lemma \ref{3.2}, there exists a positive constant $A$, which is chosen common for all $I_i$, such that
$$ \|\tilde f (z)\|^d\le A\max_{0\le j\le l_i}|Q_{I_i(j)}(\tilde f)(z)|\ \forall z\in S(i). $$
Therefore, for $z\in S(i)$, By Lemma \ref{3.1} we have
\begin{align*}
\prod_{i=1}^q\dfrac{\|\tilde f (z)\|^d}{|Q_i(\tilde f)(z)|}&\le A^{q-l_i}\prod_{j=0}^{l_j-1}\dfrac{\|\tilde f (z)\|^d}{|Q_{I_i(j)}(\tilde f)(z)|}\\
&\le A^{q-l_j}\prod_{j=0}^{n-1}\left(\dfrac{\|\tilde f (z)\|^d}{|Q_{I_i(t_j)}(\tilde f)(z)|}\right)^{t_{i,j+1}-t_{i,j}}\\
&\le A^{q-l_j}\prod_{j=0}^{n-1}\left(\dfrac{\|\tilde f (z)\|^d}{|Q_{I_i(t_j)}(\tilde f)(z)|}\right)^{\Delta}\\
&\le A^{q-l_j}B^{n\Delta}\prod_{j=0}^{n-1}\left(\dfrac{\|\tilde f (z)\|^d}{|P_{i,j}(\tilde f)(z)|}\right)^{\Delta}\\
&\le A^{q-l_j}B^{n\Delta}\prod_{j=0}^{n}\left(\dfrac{\|\tilde f (z)\|^d}{|P_{i,j}(\tilde f)(z)|}\right)^{\Delta},
\end{align*}

Consider the mapping $\Phi$ from $V$ into $\P^{l-1}(\C)\ (l=n_0(n+1))$, which maps a point ${\bf x}=(x_0:\cdots:x_N)\in V$ into the point $\Phi({\bf x})\in\P^{l-1}(\C)$ given by
$$\Phi({\bf x})=(P_{1,0}(x):\cdots : P_{1,n}(x):P_{2,0}(x):\cdots:P_{2,n}(x):\cdots:P_{n_0,0}(x):\cdots :P_{n_0,n}(x)),$$
where $x=(x_0,\ldots,x_N)$. 
Let $Y=\Phi (V)$. Since $V\cap\bigcap_{j=0}^{n}P_{1,j}^*=\varnothing$, $\Phi$ is a finite morphism on $V$ and $Y$ is a complex projective subvariety of $\P^{l-1}(\C)$ with $\dim Y=n$ and 
$$\delta:=\deg Y\le d^n.\deg V.$$ 
For every 
$${\bf a} = (a_{1,0},\ldots,a_{1,n},a_{2,0}\ldots,a_{2,n},\ldots,a_{n_0,0},\ldots,a_{n_0,n})\in\mathbb Z^l_{\ge 0}$$ 
and
$${\bf y} = (y_{1,0},\ldots,y_{1,n},y_{2,0}\ldots,y_{2,n},\ldots,y_{n_0,0},\ldots,y_{n_0,n})$$ 
we denote ${\bf y}^{\bf a} = y_{1,0}^{a_{1,0}}\ldots y_{1,n}^{a_{1,n}}\ldots y_{n_0,0}^{a_{n_0,0}}\ldots y_{n_0,n}^{a_{n_0,n}}$. Let $u$ be a positive integer. We set
\begin{align*}
n_u:=H_Y(u)-1,\ \xi_u:=\binom{l+u-1}{u}-1,
\end{align*}
and define
$$ Y_u:=\C[y_1,\ldots,y_l]_u/(I_Y)_u, $$
which is a vector space of dimension $n_u+1$. Fix a basis $\{v_0,\ldots, v_{n_u}\}$ of $Y_u$ and consider the meromorphic mapping $F$ into $\P^{n_u}(\C)$ with the following  reduced representation
$$ \tilde F=(v_0(\Phi\circ \tilde f),\ldots,v_{n_u}(\Phi\circ \tilde f)). $$
Hence $F$ is linearly nondegenerate, since $f$ is algebraically nondegenerate.

Now, we fix an index $i\in\{1,\ldots,n_0\}$ and a point $z\in S(i)$. We define 
$${\bf c}_z = (c_{1,0,z},\ldots,c_{1,n,z},c_{2,0,z},\ldots,c_{2,n,z},\ldots,c_{n_0,0,z},\ldots,c_{n_0,n,z})\in\mathbb Z^{l},$$ 
where
\begin{align}\label{4.1}
c_{i,j,z}:=\log\frac{\|\tilde f(z)\|^d\|P_{i_0,j}\|}{|P_{i,j}(\tilde f)(z)|}\text{ for } i=1,\ldots,n_0 \text{ and }j=0,\ldots,n.
\end{align}
We see that $c_{i,j,z}\ge 0$ for all $i$ and $j$. By the definition of the Hilbert weight, there are ${\bf a}_{1,z},\ldots,{\bf a}_{H_Y(u),z}\in\mathbb N^{l}$ with
$$ {\bf a}_{i,z}=(a_{i,1,0,z},\ldots,a_{i,1,n,z},\ldots,a_{i,n_0,0,z},\ldots,a_{i,n_0,n,z}), $$
where $a_{i,j,s,z}\in\{1,\ldots,\xi_u\},$ such that the residue classes modulo $(I_Y)_u$ of ${\bf y}^{{\bf a}_{1,z}},\ldots,{\bf y}^{{\bf a}_{H_Y(u),z}}$ form a basic of $\C[y_1,\ldots,y_p]_u/(I_Y)_u$ and
\begin{align}\label{4.2}
S_Y(u,{\bf c}_z)=\sum_{i=1}^{H_Y(u)}{\bf a}_{i,z}\cdot{\bf c}_z.
\end{align}
We see that ${\bf y}^{{\bf a}_{i,z}}\in Y_u$ (modulo $(I_Y)_u$). Then we may write
$$ {\bf y}^{{\bf a}_{i,z}}=L_{i,z}(v_0,\ldots,v_{H_Y(u)}), $$ 
where $L_{i,z}\ (1\le i\le H_Y(u))$ are independent linear forms.
We have
\begin{align*}
\log\prod_{i=1}^{H_Y(u)} |L_{i,z}(\tilde F(z))|&=\log\prod_{i=1}^{H_Y(u)}\prod_{\overset{1\le t\le n_0}{0\le j\le n}}|P_{t,j}(\tilde f(z))|^{a_{i,t,j,z}}\\
&=-S_Y(u,{\bf c}_z)+duH_Y(u)\log \|\tilde f(z)\| +O(uH_Y(u)).
\end{align*}
It follows that
\begin{align*}
\log\prod_{i=1}^{H_Y(u)}\dfrac{\|\tilde F(z)\|\cdot \|L_{i,z}\|}{|L_{i,z}(\tilde F(z))|}=&S_Y(u,{\bf c}_z)-duH_Y(u)\log \|\tilde f(z)\| \\
&+H_Y(u)\log \|\tilde F(z)\|+O(uH_Y(u)).
\end{align*}
Note that $L_{i,z}$ depends on $i$ and $z$, but the number of these linear forms is finite. Denote by $\mathcal L$ the set of all $L_{i,z}$ occurring in the above inequalities. We have
\begin{align}\label{4.3}
\begin{split}
S_Y(u,{\bf c}_z)\le&\max_{\mathcal J\subset\mathcal L}\log\prod_{L\in \mathcal J}\dfrac{\|\tilde F(z)\|\cdot \|L\|}{|L(\tilde F(z))|}+duH_Y(u)\log \|\tilde f(z)\|\\
& -H_Y(u)\log \|\tilde F(z)\|+O(uH_Y(u)),
\end{split}
\end{align}
where the maximum is taken over all subsets $\mathcal J\subset\mathcal L$ with $\sharp\mathcal J=H_Y(u)$ and $\{L;L\in\mathcal J\}$ is linearly independent.
From Theorem \ref{2.12} we have
\begin{align}\label{4.4}
\dfrac{1}{uH_Y(u)}S_Y(u,{\bf c}_z)\ge&\frac{1}{(n+1)\delta}e_Y({\bf c}_z)-\frac{(2n+1)\delta}{u}\max_{\underset{0\le j\le n}{1\le i\le n_0}}c_{i,j,z}
\end{align}
We chose an index $i_0$ such that $z\in S(i_0)$. It is clear that
\begin{align*}
\max_{\underset{0\le j\le n}{1\le i\le n_0}}c_{i,j,z}\le \sum_{0\le j\le n}\log\frac{\|\tilde f(z)\|^d\|P_{i_0,j}\|}{|P_{i_0,j}(\tilde f)(z)|}+O(1),
\end{align*}
where the term $O(1)$ does not depend on $z$ and $i_0$.
Combining (\ref{4.3}), (\ref{4.4}) and the above remark, we get
\begin{align}\nonumber
\frac{1}{(n+1)\delta}e_Y({\bf c}_z)\le &\dfrac{1}{uH_Y(u)}\left (\max_{\mathcal J\subset\mathcal L}\log\prod_{L\in \mathcal J}\dfrac{\|\tilde F(z)\|\cdot \|L\|}{|L(\tilde F(z))|}-H_Y(u)\log \|\tilde F(z)\|\right )\\
\label{4.5}
\begin{split}
&+d\log \|\tilde f(z)\|+\frac{(2n+1)\delta}{u}\max_{\underset{0\le j\le n}{1\le i\le n_0}}c_{i,j,z}+O(1/u)\\
\le &\dfrac{1}{uH_Y(u)}\left (\max_{\mathcal J\subset\mathcal L}\prod_{L\in\mathcal J}\dfrac{\|\tilde F(z)\|\cdot \|L\|}{|L(\tilde F(z))|}-H_Y(u)\log \|\tilde F(z)\|\right )\\
&+d\log \|\tilde f(z)\|+\frac{(2n+1)\delta}{m}\sum_{0\le j\le n}\log\frac{\|\tilde f(z)\|^d\|P_{i_0,j}\|}{|P_{i_0,j}(\tilde f)(z)|}+O(1/u).
\end{split}
\end{align}
Since $\{P_{i_0,0}=\ldots=P_{i_0,n}=0\}\cap V=\varnothing$, by Lemma \ref{2.13}, we have
\begin{align}\label{4.6}
e_Y({\bf c}_z)\ge (c_{i_0,0,z}+\cdots +c_{i_0,n,z})\cdot\delta =\left (\sum_{0\le j\le n}\log\frac{\|\tilde f(z)\|^d\|P_{i_0,j}\|}{|P_{i_0,j}(\tilde f)(z)|}\right )\cdot\delta.
\end{align}
Then, from (\ref{4.1}), (\ref{4.5}) and (\ref{4.6}) we have
\begin{align}\label{4.7}
\begin{split}
\frac{1}{\Delta}&\log \prod_{i=1}^q\dfrac{\|\tilde f (z)\|^d}{|Q_i(\tilde f)(z)|}\le\dfrac{n+1}{uH_Y(u)}\left (\max_{\mathcal J\subset\mathcal L}\log\prod_{L\in\mathcal J}\dfrac{\|\tilde F(z)\|\cdot \|L\|}{|L(\tilde F(z))|}-H_Y(u)\log \|\tilde F(z)\|\right )\\
&+d(n+1)\log \|\tilde f(z)\|+\frac{(2n+1)(n+1)\delta}{u}\sum_{\underset{0\le j\le n}{1\le i\le n_0}}\log\frac{\|\tilde f(z)\|^d\|P_{i,j}\|}{|P_{i,j}(\tilde f)(z)|}+O(1),
\end{split}
\end{align}
where the term $O(1)$ does not depend on $z$. 

Integrating both sides of the above inequality, we obtain 
\begin{align}\nonumber
\frac{1}{d}\sum_{i=1}^qm_f(r,Q_i)\le& \dfrac{\Delta(n+1)}{duH_Y(u)}\left (\int\limits_{S(r)}\max_{\mathcal J\subset\mathcal L}\log\prod_{L\in\mathcal J}\dfrac{\|\tilde F(z)\|\cdot \|L\|}{|L(\tilde F(z))|}\sigma_m-H_Y(u)T_F(r)\right )\\
\label{4.8}
&+\Delta(n+1)T_f(r)+\frac{\Delta(2n+1)(n+1)\delta}{du}\sum_{{\underset{0\le j\le n}{1\le i\le n_0}}}m_f(r,P_{i,j}).
\end{align} 
On the other hand, by Theorem \ref{2.4}, for every $\epsilon'>0$ (which will be chosen later) we have
\begin{align}
\label{4.9}
\begin{split}
\biggl \|\ \int\limits_{S(r)}\max_{\mathcal J\subset\mathcal L}&\log\prod_{L\in\mathcal J}\dfrac{\|\tilde F(z)\|\cdot \|L\|}{|L(\tilde F(z))|}\sigma_m-H_Y(u)T_F(r)\\
&\le -N_{W(\tilde F)}(r)+\epsilon' T_F(r)\le  -N_{W(\tilde F)}(r)+\epsilon' du T_f(r).
\end{split}
\end{align}
Combining this inequality with (\ref{4.8}), we have
\begin{align}\label{4.10}
\begin{split}
\bigl (q-\Delta(n+1)\bigl )T_f(r)\le& \sum_{i=1}^q\frac{1}{d}N_{Q_i(\tilde f)}(r)-\dfrac{\Delta(n+1)}{duH_Y(u)}(N_{W(\tilde F)}(r)-\epsilon' du T_f(r))\\
&+\frac{\Delta(2n+1)(n+1)\delta}{ud}\sum_{{\underset{0\le j\le n}{1\le i\le n_0}}}m_f(r,P_{i,j})
\end{split}
\end{align}

We now estimate the quantity $N_{W(\tilde F)}(r)$. Consider a point $z\in\C^m$ outside the indeterminacy locus of $f$. Without loss of generality, we may assume that $z\in S(1)$, where $I_1=(1,\ldots,q)$ and moreover
$$ \nu_{(Q_1(\tilde f))}(z)\ge \nu_{(Q_2(\tilde f))}(z)\ge\cdots\ge \nu_{(Q_q(\tilde f))}(z).$$
Since $\bigcap_{j=1}^{l_1+1}Q_j^*\cap V=\varnothing$, then $\nu^0_{Q_i(\tilde f)}(z)=0$ for all $i\ge l_1+1$. Set 
$$c_{i,j}=\max\{0,\nu^0_{P_{i,j}}(z)-H_Y(u)\}$$
and 
$${\bf c}=(c_{1,0},\ldots,c_{1,n},\ldots,c_{n_0,0},\ldots,c_{n_0,n})\in\mathbb Z^l_{\ge 0}.$$
Then there are 
$${\bf a}_i=(a_{i,1,0},\ldots,a_{i,1,n},\ldots,a_{i,n_0,0},\ldots,a_{i,n_0,n}),a_{i,j,s}\in\{1,\ldots,\xi_u\}$$
such that ${\bf y}^{{\bf a}_1},\ldots,{\bf y}^{{\bf a}_{H_Y(u)}}$ is a basic of $Y_u$ and
$$ S_Y(u,{\bf c})=\sum_{i=1}^{H_Y(u)}{\bf a}_i{\bf c}.$$
Similarly as above, we write ${\bf y}^{{\bf a}_i}=L_i(v_1,\ldots,v_{H_Y(u)})$, where $L_1,\ldots,L_{H_Y(u)}$ are independent linear forms in variables $y_{i,j}\ (1\le i\le n_0,0\le j\le n)$. From the property of the general Wronskian, we have
$$W(\tilde F)=cW(L_1(\tilde F),\ldots,L_{H_Y(u)}(\tilde F)),$$
where $c$ is a nonzero constant. This yields that
$$ \nu^0_{W(\tilde F)}(z)=\nu^0_{W(L_1(\tilde f),\ldots,L_{H_Y(u)}(\tilde F))}\ge\sum_{i=1}^{H_Y(u)}\max\{0,\nu^0_{L_i(\tilde F)}(z)-n_u\}.$$
Then
$$ \nu^0_{L_i(\tilde F)}(z)=\sum_{\underset{0\le s\le n}{1\le j\le n_0}}a_{i,j,s}\nu^0_{P_{j,s}(\tilde f)}(z), $$
and hence
$$ \max\{0,\nu^0_{L_i(\tilde F)}(z)-n_u\}\ge\sum_{i=1}^{H_Y(u)}a_{i,j,s}c_{j,s}={{\bf a}_i}\cdot{\bf c}. $$
Thus, 
\begin{align}\label{4.11}
 \nu^0_{W(\tilde F)}(z)\ge\sum_{i=1}^{H_Y(u)}{{\bf a}_i}\cdot{\bf c}=S_Y(u,{\bf c}).
\end{align}
Since $\{P_{1,0}=P_{1,1}=\cdots=P_{1,n}=0\}\cap V=\varnothing$, then by Lemma \ref{2.13} we have
$$ e_Y({\bf c})\ge \delta\cdot\sum_{j=0}^{n}c_{1,j}=\delta\cdot\sum_{j=0}^{n}\max\{0,\nu^0_{P_{1,j}(\tilde f)}(z)-n_u\}. $$
On the other hand, by Theorem \ref{2.12} we have
\begin{align*}
 S_Y(u,{\bf c}) &\ge\frac{uH_Y(u)}{(n+1)\delta}e_Y({\bf c})-(2n+1)\delta H_Y(u)\max_{\underset{0\le j\le n}{1\le i\le n_0}}c_{i,j}\\
&\ge \frac{uH_Y(u)}{n+1}\sum_{j=0}^{n+1}\max\{0,\nu^0_{P_{1,j}(\tilde f)}(z)-n_u\}-(2n+1)\delta H_Y(u)\max_{\underset{0\le j\le n}{1\le i\le n_0}}\nu^0_{P_{i,j}(\tilde f)}(z).
\end{align*}
Combining this inequality and (\ref{4.11}), we have
\begin{align}\label{4.12}
\begin{split}
\dfrac{n+1}{duH_Y(u)}\nu^0_{W(\tilde F)}(z)\ge&\frac{1}{d}\sum_{j=0}^{n}\max\{0,\nu^0_{P_{1,j}(\tilde f)}(z)-n_u\}\\
&-\frac{(2n+1)(n+1)\delta}{du}\max_{\underset{0\le j\le n}{1\le i\le n_0}}\nu^0_{P_{i,j}(\tilde f)}(z).
\end{split}
\end{align}
Also it is easy to see that $\nu^0_{P_{1,j}(\tilde f)}(z)\ge \nu^0_{Q_{I_1(t_{1,j})}(\tilde f)}(z)$ for all $0\le j\le n$ (note that $I_1(t_{1,j})=t_{1,j}+1$, $P_{1,0}=Q_1$). Therefore, we have
\begin{align*}
\Delta&\sum_{j=0}^{n}\max\{0,\nu^0_{P_{1,j}(\tilde f)}(z)-n_u\}\\
&\ge \Delta\sum_{j=0}^{n}\max\{0,\nu^0_{Q_{I_1(t_{1,j})}(\tilde f)}(z)-n_u\}\\
&\ge\sum_{j=0}^{n}(t_{j+1}-t_j)\max\{0,\nu^0_{Q_{I_1(t_{1,j})}(\tilde f)}(z)-n_u\}\\
&\ge\sum_{j=0}^{l_1}\max\{0,\nu^0_{Q_{I_1(j)}(\tilde f)}(z)-n_u\}\\
&=\sum_{i=1}^{q}\max\{0,\nu^0_{Q_i(\tilde f)}(z)-n_u\}.
\end{align*}
Combining this inequality and (\ref{4.12}), we have
\begin{align}\label{4.13}
\begin{split}
\dfrac{\Delta(n+1)}{duH_Y(u)}\nu^0_{W(\tilde F)}(z)\ge&\frac{1}{d}\sum_{i=1}^{q}\max\{0,\nu^0_{Q_i(\tilde f)}(z)-n_u\}\\
&-\frac{\Delta(2n+1)(n+1)\delta}{du}\max_{\underset{0\le j\le n}{1\le i\le n_0}}\nu^0_{P_{i,j}(\tilde f)}(z)\\
\ge& \frac{1}{d}\sum_{i=1}^{q}(\nu^0_{Q_i(\tilde f)}(z)-\min\{\nu^0_{Q_i(\tilde f)}(z),u\})\\
&-\frac{\Delta(2n+1)(n+1)\delta}{du}\max_{\underset{0\le j\le n}{1\le i\le n_0}}\nu^0_{P_{i,j}(\tilde f)}(z).
\end{split}
\end{align}
Integrating both sides of this inequality, we get
\begin{align}\label{4.14}
\begin{split}
\dfrac{\Delta(n+1)}{duH_Y(u)}N_{W(\tilde F)}(r)&\ge\frac{1}{d}\sum_{i=1}^{q}(N_{Q_i(\tilde f)}(r)-N^{[n_u]}_{Q_i(\tilde f)}(r))\\
&-\frac{\Delta(2n+1)(n+1)\delta}{du}\max_{\underset{0\le j\le n}{1\le i\le n_0}}N_{P_{i,j}(\tilde f)}(r).
\end{split}
\end{align}

Combining inequalities (\ref{4.10}) and (\ref{4.14}), we get
\begin{align}\label{4.15}
\begin{split}
\bigl \|\ (q-&\Delta(n+1)\bigl )T_f(r)\le \sum_{i=1}^q\frac{1}{d}N^{[n_u]}_{Q_i(\tilde f)}(r)+\dfrac{\Delta(n+1)}{H_Y(u)}\epsilon' T_f(r)\\
&+\frac{\Delta(2n+1)(n+1)\delta}{ud}\sum_{{\underset{0\le j\le n}{1\le i\le n_0}}}(N_{P_{i,j}(\tilde f)}(r)+m_f(r,P_{i,j}))\\
=&\sum_{i=1}^q\frac{1}{d}N^{[n_u]}_{Q_i(\tilde f)}(r)+\left (\dfrac{\Delta(n+1)\epsilon'}{H_Y(u)}+\frac{\Delta(2n+1)(n+1)l\delta}{u}\right)T_f(r).
\end{split}
\end{align}

We now choose $u$ is the smallest integer such that
$$ u> \Delta(2n+1)(n+1)l\delta\epsilon^{-1} $$
and 
$$\epsilon'=\dfrac{H_Y(u)}{\Delta(n+1)}\left (\epsilon-\frac{\Delta(2n+1)(n+1)l\delta}{u}\right)>0.$$
Form (\ref{4.15}), we have
\begin{align*}
\bigl \|\ (q-\Delta(n+1)-\epsilon\bigl )T_f(r)\le \sum_{i=1}^q\frac{1}{d}N^{[n_u]}_{Q_i(\tilde f)}(r).
\end{align*}
Note that $\deg Y=\delta\le d^n\deg (V)$. Then the number $n_u$ is estimated as follows
\begin{align}\label{4.16}
\begin{split}
n_u&=H_Y(u)-1\le\delta \binom{n+u}{n}\le d^n\deg (V)e^{n}\left(1+\frac{u}{n}\right)^{n}\\
&<d^n\deg (V)e^n\left (\Delta(2n+4)l\delta\epsilon^{-1}\right)^n\\
&\le \left[\deg (V)^{n+1}e^nd^{n^2+n}\Delta^n(2n+4)^nl^n\epsilon^{-n}\right]=M_0.
\end{split}
\end{align} 
Then, the theorem is proved for the case where all hypersurfaces $Q_i$ have the same degree.

Now, for the general case where $Q_i (1\le i\le q)$ is of the degree $d_i$. Then all $Q_i^{d/d_i}$ are of the same degree $d\ (1\le i\le q)$. Applying the above result, we have
\begin{align*}
\bigl \|\ (q-\Delta(n+1)-\epsilon)T_f(r)&\le\sum_{i=1}^{q}\frac{1}{d}N^{[M_0]}_{Q^{d/d_i}(\tilde f)}(r)+o(T_f(r))\\
&\le\sum_{i=1}^{q}\frac{1}{d_i}N^{[M_0]}_{Q(\tilde f)}(r)+o(T_f(r)).
\end{align*}
The theorem is proved.
\end{proof}

\begin{proof}[{\bf Proof of Theorem \ref{1.2}}] 
We repeat almost completely the same arguments as in the proof of Theorem \ref{1.1} with only one modification that the inequality (\ref{4.9}) is replaced by the following:
\begin{align*}
\begin{split}
&\int\limits_{S(r)}\max_{\mathcal J\subset\mathcal L}\log\prod_{L\in\mathcal J}\dfrac{\|\tilde F(z)\|\cdot \|L\|}{|L(\tilde F(z))|}\sigma_m-H_Y(u)T_F(r,r_0)\\
&\le -N_{W(\tilde F)}(r,r_0)+\dfrac{(H_Y(u)-1)H_Y(u)}{2}(1+\epsilon')(c_F+\epsilon') T_F(r,r_0)+O(\log T_f(r,r_0))\\
&\le  -N_{W(\tilde F)}(r,r_0)+\dfrac{(H_Y(u)-1)H_Y(u)}{2}(1+\epsilon')(\dfrac{c_f}{du}+\epsilon') du T_f(r,r_0)+O(\log T_f(r,r_0))
\end{split}
\end{align*}
for all $r\in (0,R_0)$ outside a set $E$ with $\int_E\mathrm{exp}((c_f+\epsilon')T_f(r,r_0))dr<+\infty$.
Here we note that $T_F(r,r_0)=duT_f(r,r_0)$, and hence $c_F=\dfrac{c_f}{du}$.

Hence (\ref{4.15}) becomes to
\begin{align*}
&(q-\Delta(n+1))T_f(r,r_0)\le \sum_{i=1}^q\frac{1}{d}N^{[n_u]}_{Q_i(\tilde f)}(r,r_0)+\frac{\Delta(2n+1)(n+1)l\delta}{u}T_f(r,r_0)\\
&+\dfrac{(H_Y(u)-1)\Delta(n+1)}{2}(1+\epsilon')\left(\dfrac{c_f}{du}+\epsilon'\right)T_f(r,r_0)+O(\log T_f(r,r_0))
\end{align*}
for all $r\in (0,R_0)$ outside a set $E$ with $\int_E\mathrm{exp}((c_f+\epsilon')T_f(r,r_0))dr<+\infty$.

Choosing $u$ the smallest integer such that $ u> \Delta(2n+1)(n+1)l\delta\epsilon^{-1}$ and $\epsilon'$ a positive number with
$$\dfrac{(H_Y(u)-1)\Delta(n+1)}{2}(1+\epsilon')\left(\dfrac{c_f}{du}+\epsilon'\right)+\frac{\Delta(2n+1)(n+1)l\delta}{u}\le\dfrac{(H_Y(u)-1)\Delta(n+1)c_f}{2du}+\epsilon,$$
 we have
\begin{align*}
(q&-\Delta(n+1)-\epsilon)T_f(r,r_0)\\
&\le \sum_{i=1}^q\frac{1}{d}N^{[M_0]}_{Q_i(\tilde f)}(r)+\dfrac{(H_Y(u)-1)\Delta(n+1)c_f}{2du}T_f(r,r_0)+O(\log T_f(r,r_0))\\
&\le \sum_{i=1}^q\frac{1}{d}N^{[M_0]}_{Q_i(\tilde f)}(r)+\dfrac{M_0c_f}{2d(2n+1)(n+1)(q!)\delta}T_f(r,r_0)+O(\log T_f(r,r_0))
\end{align*}
for all $r\in (r_0,R_0)$ outside a set $E$ with $\int_E\mathrm{exp}((c_f+\epsilon')T_f(r,r_0))dr<+\infty$.
\end{proof}

\begin{proof}[{\bf Proof of Theorem \ref{1.3}}]
By using the universal covering if necessary, we may assume that $M=\mathbb{B}^m(R_0)\ (0<R_0<+\infty)$. Replacing $Q_i$ by $Q_i^{d/d_j} \ (j=1,\ldots, q)$ if necessary, we may assume that all hypersurfaces $Q_i\ (1\le i\le q)$ are of the same degree $d$. 

We use the same notation as in the proof of Theorem \ref{1.1} as follows: 
\begin{itemize}
\item $\mathcal I=\{I_1,\ldots,I_{n_0}\}$ is the set of all permutations of the set $\{1,\ldots.,q\}$, where $n_0=q!$, $I_i=(I_i(1),\ldots,I_i(q))\in\mathbb N^q$.
\item $t_{i,0},t_{i,1},\ldots,t_{i,n}$ are $n+1$ integers with $0=t_{i,0}<t_{i,1}<\cdots<t_{i,n}=l_i$, where $l_i$ is an integer, $l_i\le q-2$ such that $\bigcap_{j=0}^{l_i}Q_{I_i(j)}\cap V=\varnothing$ and
$$\dim\left (\bigcap_{j=0}^{s}Q_{I_i(j)}\right )\cap V=n-u\ \forall t_{i,u-1}\le s<t_{i,u},1\le u\le n.$$
\item $P_{i,0},\ldots,P_{i,n}$ are hypersurfaces obtained in Lemma \ref{3.1} with respect to the hypersurfaces $Q_{I_i(0)},\ldots,Q_{I_i(l_i)}$. 
\item $\Phi$ is the map from $V$ into $\P^{l-1}(\C)\ (l=n_0(n+1))$, which maps a point ${\bf x}=(x_0:\cdots:x_N)\in V$ into the point $\Phi({\bf x})\in\P^{l-1}(\C)$ given by
$$\Phi({\bf x})=(P_{1,0}(x):\cdots : P_{1,n}(x):\cdots:P_{n_0,0}(x):\cdots :P_{n_0,n}(x)),$$
where $x=(x_0,\ldots,x_N)\in\C^{N+1}.$ 
\item $Y=\Phi (V)$, $n_u:=H_Y(u)-1,\ \xi_u:=\binom{l+u-1}{u}-1$, $ Y_u=\C[y_1,\ldots,y_l]_u/(I_Y)_u$, $\{v_0,\ldots, v_{n_u}\}$ is a basis of $Y_u$ and $F$ is a linearly nondegenerate meromorphic mapping with the following reduced representation
$$ \tilde F=(v_0(\Phi\circ \tilde f),\ldots,v_{n_u}(\Phi\circ \tilde f)). $$
\end{itemize} 
Then there exists an admissible set $\alpha=(\alpha_0,\ldots,\alpha_{n_u})\in(\mathbb{Z}^m_+)^{n_u+1}$ such that
$$W^\alpha(F_0,\ldots,F_{n_u})=\det (D^{\alpha_i}(v_s(\Phi\circ \tilde f)))_{0\le i,s\le n_u}\neq 0,$$
where $F_j=v_j(\Phi\circ \tilde f)$.

Then, from (\ref{4.1}) we have
\begin{align}\label{4.17}
\begin{split}
\frac{1}{\Delta}&\log \prod_{i=1}^q\dfrac{\|\tilde f (z)\|^d}{|Q_i(\tilde f)(z)|}\le\dfrac{n+1}{uH_Y(u)}\left (\max_{\mathcal J\subset\mathcal L}\log\prod_{L\in\mathcal J}\dfrac{\|\tilde F(z)\|}{|L(\tilde F(z))|}-H_Y(u)\log \|\tilde F(z)\|\right )\\
&+d(n+1)\log \|\tilde f(z)\|+\frac{(2n+1)(n+1)\delta}{u}\sum_{\underset{1\le i\le n_0}{0\le j\le n+1}}\log\frac{\|\tilde f(z)\|^d}{|P_{i,j}(\tilde f)(z)|}+O(1),
\end{split}
\end{align}
where $O(1)$ does not depend on $z$, the maximum is taken over all subsets $\mathcal J\subset\mathcal L$ with $\sharp\mathcal J=H_Y(u)$ and $\{L;L\in\mathcal J\}$ is linearly independent, and $\mathcal L$ is defined as in the proof of Theorem \ref{1.1}. 

Set $m_0=(2n+1)(n+1)\delta$ and $b=\dfrac{n+1}{uH_Y(u)}$. There exists a positive constant $K_0$ such that
\begin{align*}
\dfrac{\|\tilde f (z)\|^{\frac{1}{\Delta}dq-d(n+1)-\frac{dm_0l}{u}}.|W^\alpha(\tilde{F}(z))|^b . \prod_{\underset{0\le j\le n}{1\le i\le n_0}}|P_{i,j}(\tilde f)(z)|^{\frac{m_0}{u}}}{\prod_{i=1}^q|Q_i(\tilde f)(z)|^{\frac{1}{\Delta}}}\le K_0^b.S_{\mathcal J}^b,
\end{align*}
where $ S_{\mathcal J}=\dfrac{|W^\alpha(\tilde{F}(z))|}{\prod_{L\in\mathcal J}|L(\tilde F(z))|}$ for some $\mathcal J \subset \mathcal{L}$ with $\# \mathcal{J} = H_Y(u)$ so that $\{L\in\mathcal J\}$ is linearly independent.

On the other hand, from (\ref{4.13}), we have
\begin{align}\label{4.18}
\dfrac{1}{\Delta}\sum_{i=1}^{q}\nu^0_{Q_i(\tilde f)}(z)-b\nu^0_{W(\tilde F)}(z)\le \dfrac{1}{\Delta}\sum_{i=1}^{q}\min\{\nu^0_{Q_i(\tilde f)}(z),n_u\}+\dfrac{m_0}{u}\max_{\underset{0\le j\le n}{1\le i\le n_0}}\nu^0_{P_{i,j}(\tilde f)}(z).
\end{align}

Assume that 
$$ \rho\Omega_f+\dfrac{\sqrt{-1}}{2\pi}\partial\bar\partial\log h^2\ge \ric\omega.$$
We now suppose that
$$ \sum_{j=1}^q\delta^{[H_Y(u)-1]}_f(Q_j)> \Delta(n+1)+\dfrac{\Delta m_0l}{u}+\dfrac{\rho\Delta H_Y(u)(H_Y(u)-1)b}{d}.$$
Then, for each $j\in\{1,\ldots ,q\},$ there exist constants $\eta_j>0$ and continuous plurisubharmonic function $\tilde u_j$ such that 
$e^{\tilde u_j}|\varphi_j|\le \|\tilde f\|^{d\eta_j},$ where $\varphi_j$ is a holomorphic function with $\nu_{\varphi_j}=\min\{n_u,f^*Q_j\}$ and
$$ q-\sum_{j=1}^q\eta_j>  \Delta(n+1)+\dfrac{\Delta m_0l}{u}+\dfrac{\rho\Delta H_Y(u)(H_Y(u)-1)b}{d}.$$
Put $u_j=\tilde u_j+\log |\varphi_j|$, then $u_j$ is a plurisubharmonic and
$$ e^{u_j}\le \|\tilde f\|^{d\eta_j},\ j=1,\ldots ,q. $$
Let
$$v (z)=\log\left |(z^{\alpha_0+\cdots+\alpha_{n_u}})^b\dfrac{(W^{\alpha}(\tilde F(z)))^b}{(\prod_{i=1}^{q}Q_i(\tilde f)(z))^{\frac{1}{\Delta}}}\right |+\dfrac{1}{\Delta}\sum_{j=1}^q u_j(z)+\dfrac{m_0}{u}\sum_{\underset{0\le j\le n}{1\le i\le n_0}}\log|P_{i,j}(\tilde f)(z)|.$$
Therefore, we have the following current inequality
\begin{align*}
2dd^c[v]&\ge b[\nu_{W^{\alpha}(\tilde F)}]-\dfrac{1}{\Delta}\sum_{j=1}^q[\nu_{Q_i(\tilde f)}]+\dfrac{1}{\Delta}\sum_{j=1}^q2dd^c[u_j]+\dfrac{m_0}{u}\sum_{\underset{1\le i\le n_0}{0\le j\le 0}}[\nu_{P_{i,j}(\tilde f)}(z)]\\
&\ge b[\nu_{W^{\alpha}(\tilde F)}]-\dfrac{1}{\Delta}\sum_{j=1}^q[\nu_{Q_i(\tilde f)}]+\dfrac{1}{\Delta}\sum_{j=1}^q[\min\{n_u,\nu_{Q_i(\tilde f)}\}]+\dfrac{m_0}{u}\max_{\underset{0\le j\le n}{1\le i\le n_0}}[\nu_{P_{i,j}(\tilde f)}(z)]\ge 0.
\end{align*}
This yields that $v$ is a plurisubharmonic function on $\B^m(R_0)$.

By the growth condition of $f$, there exists a continuous plurisubharmonic function $\omega\not\equiv\infty$ on $\B^m(R_0)$ such that
\begin{align*}
e^\omega {\rm d}V\le \|\tilde f\|^{2\rho}v_m.
\end{align*}
Set
$$t=\dfrac{2\rho}{\dfrac{1}{\Delta}d\left(q-\Delta(n+1)-\dfrac{\Delta m_0l}{u}-\sum_{j=1}^q\eta_j\right)}>0$$ 
and 
$$\lambda (z)=(z^{\alpha_0+\cdots +\alpha_{n_u}})^b\dfrac{\left (W^{\alpha}(\tilde F(z))\right )^b\cdot \prod_{\underset{0\le j\le n}{1\le i\le n_0}}|P_{i,j}(\tilde f)(z)|^{\frac{m_0}{u}}}{Q_1^{\frac{1}{\Delta}}(\tilde f)(z)\cdots Q_q^{\frac{1}{\Delta}}(\tilde f)(z)}.$$ 
Then, we see that
$$ \dfrac{H_Y(u)(H_Y(u)-1)b}{2}t< \dfrac{H_Y(u)(H_Y(u)-1)b}{2}\cdot\dfrac{2\rho}{\dfrac{1}{\Delta}d\dfrac{\rho \Delta H_Y(u)(H_Y(u)-1)b}{d}}=1,$$
and the function $\zeta=\omega+ tv$ is plurisubharmonic on $M$. Choose a positive number $\xi$ such that $0<\dfrac{H_Y(u)(H_Y(u)-1)b}{2}t<\xi<1.$
Then, we have
\begin{align}\label{4.19}
\begin{split}
e^\zeta dV&=e^{\omega +tv}dV\le e^{tv}\|\tilde f\|^{2\rho}v_m=|\lambda|^{t}(\prod_{j=1}^qe^{t\frac{1}{\Delta}u_j})\|\tilde f\|^{2\rho}v_m\\
&\le|\lambda|^t \|\tilde f\|^{2\rho+\sum_{j=1}^q\frac{1}{\Delta}dt\eta_j}v_m=|\lambda|^t \|\tilde f\|^{t\frac{1}{\Delta}d(q-\Delta(n+1)-\frac{\Delta m_0l}{u})}v_m.
\end{split}
\end{align}

We consider the following two cases.

(a) Case 1: $R_0< \infty$ and $ \lim\limits_{r\rightarrow R_0}\sup\dfrac{T_f(r,r_0)}{\log 1/(R_0-r)}<\infty.$
It suffices for us to prove the theorem in the case where $\mathbb{B}^m(R_0)=\mathbb{B}^m(1).$
Integrating both sides of (\ref{4.19}) over $\mathbb{B}^m(1),$  we have
\begin{align}\label{4.20}
\begin{split}
\int_{\B^m(1)}e^\zeta dV&\le \int_{\B^m(1)}|\lambda|^t \|\tilde f\|^{t(\frac{1}{\Delta}dq-d(n+1)-\frac{dm_0l}{u})}v_m.\\
&=2m\int_0^1r^{2m-1}\left (\int_{S(r)}\bigl (|\lambda| \|\tilde f\|^{\frac{1}{\Delta}dq-d(n+1)-\frac{dm_0l}{u}}\bigl )^t\sigma_m\right )dr\\
&\le 2m\int_0^1r^{2m-1}\left (\int_{S(r)}\sum_{\mathcal J}\bigl |(z^{\alpha_0+\cdots +\alpha_{n_u}})K_0S_{\mathcal J}\bigl |^{bt}\sigma_m\right )dr,
\end{split}
\end{align}
where the summation is taken over all $\mathcal J\subset\mathcal L$ with $\sharp J=H_Y(u)$ and $\{L\in\mathcal J\}$ is linearly independent.

Note that $(\sum_{i=0}^{n_u}|\alpha_i|)bt\le \dfrac{H_Y(u)(H_Y(u)-1)b}{2}t<\xi<1$. Then by Proposition \ref{pro2.2} there exists a positive constant $K_1$ such that, for every $0<r_0<r<r'<1,$ 
\begin{align*}
\int_{S(r)}\left |(z^{\alpha_0+\cdots +\alpha_{n_u}})K_0S_{\mathcal J}(z)\right |^{bt}\sigma_m\le K_1\left (\dfrac{r'^{2m-1}}{r'-r}dT_f(r',r_0)\right )^{\xi}.
\end{align*}
Choosing $r'=r+\dfrac{1-r}{eT_f(r,r_0)}$, we get
$$ T_f(r',r_0)\le 2T_f(r,r_0)$$
outside a subset $E\subset [0,1]$ with $\int_E\dfrac{dr}{1-r}<+\infty$. Hence, the above inequality implies that
\begin{align*}
\int_{S(r)}\left |(z^{\alpha_1+\cdots +\alpha_{n_u}})K_0S_J(z)\right |^{bt}\sigma_m\le \dfrac{K}{(1-r)^\xi}\left (\log\dfrac{1}{1-r}\right )^{\xi}
\end{align*}
for all $r$ outside $E$, where $K$ is a positive constant. By choosing $K$ large enough, we may assume that the above inequality holds for all $r\in (0;1)$.
Then, the inequality (\ref{4.20}) yields that
\begin{align*}
\int_{\B^m(1)}e^\zeta dV&\le 2m\int_0^1r^{2m-1}\dfrac{K}{(1-r)^\delta}\left (\log\dfrac{1}{1-r}\right )^{\delta}dr< +\infty
\end{align*}
This contradicts the results of S. T. Yau  and L. Karp (see \cite{K82,Y76}). 

Therefore, we must have
\begin{align*}
\sum_{j=1}^q\delta^{[H_Y(u)-1]}_{f}(Q_j)&\le \Delta(n+1)+\dfrac{\Delta m_0 l}{u}+\dfrac{\rho \Delta H_Y(u)(H_Y(u)-1)b}{d}\\
&= \Delta(n+1)+\dfrac{\Delta m_0 l}{u}+\dfrac{\rho \Delta (n+1)(H_Y(u)-1)}{ud}.
\end{align*}
Choosing $u$ be the smallest integer such that $\epsilon -\dfrac{\Delta m_0l}{u}\ge 0$, i.e., $u\ge\Delta (2n+1)(n+1)l\delta\epsilon^{-1}$, we have
$$\sum_{j=1}^q\delta^{[H_Y(u)-1]}_{f}(Q_j)\le \Delta (n+1)+\epsilon+\dfrac{\rho\epsilon (H_Y(u)-1)}{\delta(2n+1)(n+1)(q!)d}.  $$
From (\ref{4.16}), we have $H_Y(u)-1\le M_0$ and also
$$\sum_{j=1}^q\delta^{[M_0]}_{f}(Q_j)\le \Delta (n+1)+\epsilon+\dfrac{\rho\epsilon M_0}{\delta(2n+1)(n+1)(q!)d}.$$
Then theorem is proved in this case.

(b) Case 2: $\lim\limits_{r\rightarrow R_0}\sup\dfrac{T(r,r_0)}{\log 1/(R_0-r)}= \infty .$
Repeating the argument in the proof of Theorem \ref{1.1}, we only need to prove the following theorem.
\begin{theorem}
With the assumption of Theorem \ref{1.3}. Then, we have
$$(q-\Delta(n+1)-\epsilon)T_f(r,r_0)\le \sum_{i=1}^{q}\dfrac{1}{d}N^{[M_0]}_{Q_i(\tilde f)}(r)+S(r),$$
where $S(r)$ is evaluated as follows:

(i) In the case $R_0<\infty,$ $$S(r)\le K(\log^+\dfrac{1}{R_0-r}+\log^+T_f(r,r_0))$$
for all $0<r_0<r<R_0$ outside a set $E\subset [0,R_0]$ with $\int_E\dfrac{dt}{R_0-t}<\infty$ and $K$ is a positive constant.

(ii) In the case  $R_0=\infty,$ $$S(r)\le K(\log r+\log^+T_f(r,r_0))$$
for all $0<r_0<r<\infty$ outside a set $E'\subset [0,\infty]$ with $\int_{E'} dt<\infty$ and $K$ is a positive constant.
\end{theorem}
	
\noindent\textit{Proof.} Repeating the above argument, we have
\begin{align*}
\int\limits_{S(r)}&\left|(z^{\alpha_0+\cdots +\alpha_{n_u}})^b\dfrac{\|\tilde f(z)\|^{\frac{1}{\Delta}qd-d(n+1)-\frac{dm_0l}{u}}|W^{\alpha}(\tilde F)(z)|^b\cdot \prod_{\underset{0\le j\le n}{1\le i\le n_0}}|P_{i,j}(\tilde f)(z)|^{\frac{m_0}{u}}}{\prod_{i=1}^{q}|Q_i(\tilde f)(z)|^{\frac{1}{\Delta}}}\right|^{t}\sigma_m\\
&\le K_1\left(\dfrac{R^{2m-1}}{R-r}dT_f(R,r_0)\right)^{\delta}.
\end{align*}
for every $0<r_0<r<R<R_0$. By the concavity of the logarithmic function, we have
\begin{align*}
\begin{split}
b&\int_{S(r)}\log |(z^{\alpha_0+\cdots +\alpha_{n_u}})|\sigma_m+\left(\dfrac{1}{\Delta}qd-d(n+1)-\dfrac{dm_0l}{u}\right)\int_{S(r)}\log \|\tilde f\|\sigma_m\\
&+b\int_{S(r)}\log |W^{\alpha}(\tilde F)|\sigma_m+\dfrac{m_0}{u}\sum_{\underset{1\le i\le n_0}{0\le j\le n}} \int_{S(r)}\log |P_{i,j}(\tilde f)|\sigma_m\\
&-\dfrac{1}{\Delta}\sum_{j=1}^q \int_{S(r)}\log |Q_j(\tilde f)|\sigma_m\le K\left(\log^+\dfrac{R}{R-r}+\log^+T_f(R,r_0)\right)
\end{split}
\end{align*}
for some positive constant $K.$ By Jensen's formula, this inequality implies that
\begin{align}\label{4.22}
\begin{split}
\biggl(\dfrac{1}{\Delta}qd&-d(n+1)-\dfrac{dm_0l}{u}\biggl)T_f(r,r_0)+bN_{W^{\alpha}(\tilde F)}(r)+\dfrac{m_0}{u}\sum_{\underset{1\le i\le n_0}{0\le j\le n+1}}N_{P_{i,j}(\tilde f)}(r)\\
&-\dfrac{1}{\Delta}\sum_{i=1}^qN_{Q_i(\tilde f)}(r)\le K\left(\log^+\dfrac{R}{R-r}+\log^+T_f(R,r_0)\right)+O(1).
\end{split}
\end{align}
From (\ref{4.10}), we have 
$$ \dfrac{1}{\Delta}\sum_{i=1}^qN_{Q_i(\tilde f)}(r)-bN_{W^{\alpha}(\tilde F)}(r)-\dfrac{m_0}{u}\sum_{\underset{1\le i\le n_0}{0\le j\le n+1}}N_{P_{i,j}(\tilde f)}(r)\le\dfrac{1}{\Delta}\sum_{i=1}^qN^{[M_0]}_{Q_i(\tilde f)}(r).$$
Combining this estimate and (\ref{4.22}), we get
\begin{align}\label{4.23}
\begin{split}
\left(q-\Delta(n+1)-\epsilon\right)T_f(r,r_0)\le&\sum_{i=1}^{q}\dfrac{1}{d}N^{[M_0]}_{Q_i(\tilde f)}(r)\\
&+K\left(\log^+\dfrac{R}{R-r}+\log^+T_f(R,r_0)\right)+O(1).
\end{split}
\end{align}
Choosing $R=r+\dfrac{R_0-r}{eT_f(r,r_0)}$ if $R_0<\infty$ and $R=r+\dfrac{1}{T_f(r,r_0)}$ if $R_0=\infty$, we see that
$$ T_f\left(r+\dfrac{R_0-r}{eT_f(r,r_0)},r_0\right)\le 2T_f(r,r_0)$$
outside a subset $E\subset [0,R_0)$ with $\int_E\dfrac{dr}{R_0-r}<+\infty$ in the case $R_0<\infty$ and  
$$ T_f\left(r+\dfrac{1}{T_f(r,r_0)},r_0\right)\le 2T_f(r,r_0)$$
outside a subset $E'\subset [0,\infty)$ with $\int_{E'}dr<\infty$ in the case $R_0=\infty$.
Thus, from (\ref{4.23}) we have
$$ (q-\Delta(n+1)-\epsilon) T_f(r,r_0)\le\sum_{i=1}^{q}\dfrac{1}{d}N^{[M_0]}_{Q_i(\tilde f)}(r)+S(r). $$
This implies that
$$\sum_{j=1}^q\delta^{[M_0]}_{f}(Q_j)\le\sum_{j=1}^q\delta^{[M_0]}_{f,*}(Q_j)\le \Delta(n+1)+\epsilon.$$
The theorem is proved in this case.
\end{proof}

\section{Generalization of subspace theorem}

Let $k$ be a number field. Denote by $M_k$ the set of places of $k$ and by $M^\infty_k$ the set of archimedean places of $k$. For each $v\in M_k$, we choose the normalized absolute value $|\cdot |_v$ such that $|\cdot |_v=|\cdot|$ on $\mathbb Q$ (the standard absolute value) if $v$ is archimedean, and $|p|_v=p^{-1}$ if $v$ is non-archimedean and lies above the rational prime $p$. For each $v\in M_k$, denote by $k_v$ the completion of $k$ with respect to $v$ and set $$n_v := [k_v :\mathbb Q_v]/[k :\mathbb Q].$$
We put $\|x\|_v=|x|^{n_v}_v$. Then, the product formula is stated as follows:
$$\prod_{v\in M_k}\|x\|_v=1,\text{ for }  x\in k^*.$$
Let $S$ be a finite subset of $M_k$ containing $M^{\infty}_k$. An element $x\in k$ is said to be an $S$-integer if $\|x\|_v\le 1$ for every $v\in M_k\setminus S$. Denote by $\mathcal O_S$ the set of all $S$-integers. 

For ${\bf x} = (x_0 ,\ldots , x_N)\in k^{N+1}$, define
$$\|{\bf x}\|_v :=\max\{\|x_0\|_v,\ldots,\|x_N\|_v\},\ v\in M_k.$$
The absolute logarithmic height of a point ${\bf x}=(x_0:\cdots :x_N)\in\P^N(k)$ is defined by
$$h({\bf x}):=\sum_{v\in M_k}\log \|{\bf x}\|_v.$$
By the product formula, this definition does not depend on the choice of homogeneous coordinates $(x_0:\cdots :x_N)$ of ${\bf x}$. If $x\in k^*$, we define the absolute logarithmic height of $x$ by
$$h(x):=\sum_{v\in M_k}\log^+ \|x\|_v,$$
where $\log^+a=\log\max\{1,a\}.$

Let $Q=\sum_{I\in\mathcal T_d}a_I{\bf x}^I$ be a homogeneous polynomial of degree $d$ in $k[x_0,\ldots, x_N]$, where ${\bf x}^I = x^{i_0}\ldots x^{i_N}_N$ for ${\bf x}=(x_0,\ldots, x_N)$  and $I = (i_0,\ldots,i_N)$. Define $\|Q\|_v =\max\{\|a_I\|_v; I\in\mathcal T_d\}$. For each ${\bf x}=(x_0,\ldots, x_N)$, we have
\begin{align*}
\|Q({\bf x})\|_v\le C\cdot\|Q\|_v\cdot \|{\bf x}\|_v
\end{align*}
for all $v$, where $C$ is a positive constant (depending only on the degree of $Q$). The height of $Q$ is defined by
$$h(Q)=\sum_{v\in M_k}\log \|Q\|_v.$$
For each $v\in M_k$, we define the Weil function $\lambda_{Q,v}$ by
$$\lambda_{Q,v}({\bf x}):=\log\frac{\|{\bf x}\|_v^d\cdot \|Q\|_v}{\|Q({\bf x})\|_v},\ {\bf x}\in\P^N(k)\setminus\{Q=0\}.$$ 

\begin{proof}[{\bf Proof of Theorem \ref{1.5}}]
Similar as above, we may assume that all $Q_i\ (1\le i\le q)$ are of the same degree $d$ and $q>\Delta(n+1)$. 
We use the same following notation as in the proofs of Theorem \ref{1.1} and Theorem \ref{1.3}: 
\begin{itemize}
\item $\mathcal I=\{I_1,\ldots,I_{n_0}\}$ is the set of all permutations of the set $\{1,\ldots.,q\}$, where $n_0=q!$, $I_i=(I_i(1),\ldots,I_i(q))\in\mathbb N^q$.
\item $t_{i,0},t_{i,1},\ldots,t_{i,n}$ are $n+1$ integers with $0=t_{i,0}<t_{i,1}<\cdots<t_{i,n}=l_i$, where $l_i\le q-2$ such that $\bigcap_{j=0}^{l_i}Q_{I_i(j)}^*\cap V(\bar k)=\varnothing$ and
$$\dim\left (\bigcap_{j=0}^{s}Q_{I_i(j)}^*\right )\cap V(\bar k)=n-u\ \forall t_{i,u-1}\le s<t_{i,u},1\le u\le n.$$
\item $P_{i,0},\ldots,P_{i,n}$ are hypersurfaces obtained in Lemma \ref{3.1} with respect to hypersurfaces $Q_{I_i(0)},\ldots,Q_{I_i(l_i)}$. 
\item $\Phi$ is the morphism from $V(\bar k)$ into $\P^{l-1}(\bar k)\ (l=n_0(n+1))$, which maps a point $x\in V(\bar k)$ into the point $\Phi(x)\in\P^{l-1}(\bar k)$ given by
$$\Phi(x)=(P_{1,0}(x):\cdots : P_{1,n}(x):\cdots:P_{n_0,0}(x):\cdots :P_{n_0,n}(x)).$$ 
\item $Y=\Phi (V)$ is a subvariety of $\P^{l-1}(k)$ with $\dim Y=n$ and $\Delta:=\deg Y\le d^n.\deg V$ (since $V(\bar k)\cap\bigcap_{j=0}^{n}P_{1,j}^*=\varnothing$, $\Phi$ is a generically finite morphism on $V$).
\item $n_u:=H_Y(u)-1,\ \xi_u:=\binom{l+u-1}{u}-1$, $ Y_u=k[y_1,\ldots,y_l]_u/(I_Y)_u$, $\{v_0,\ldots, v_{n_u}\}$ is a basis of $Y_u$ so that each $v_i$ is a residue class of a monic monomial in $k[y_1,\ldots,y_l]_u$.
\end{itemize} 
For every 
$${\bf a} = (a_{1,0},\ldots ,a_{1,n},\ldots,a_{n_0,0},\ldots,a_{n_0,n})\in\mathbb Z^l_{\ge 0}$$ 
and
$${\bf y} = (y_{1,0},\ldots ,y_{1,n},\ldots,y_{n_0,0},\ldots,y_{n_0,n})$$ 
we denote ${\bf y}^{\bf a} = y_{1,0}^{a_{1,0}}\ldots y_{n_0,0}^{a_{n_0,0}}\ldots y_{n_0,n}^{a_{n_0,n}}$. 
Consider the morphism $F$ from $V$ into $\P^{H_Y(u)-1}(\bar k)$ with the following presentation
$$ F({\bf x})=(v_1(\Phi({\bf x})):\cdots :v_{H_Y(u)}(\Phi({\bf x}))),\ \ {\bf x}\in V(\bar k).$$

Since $S$ is finite set and the number of hypersurfaces occurring in this proof is finite, there is a positive constant $c$ such that for all $v\in S$ and a given hypersurfaces $Q\in\bar k[x_0,\ldots,x_N]$,
$$c\|{\bf x}\|_v\ge |Q({\bf x})| \ \forall {\bf x}=(x_0,\ldots,x_N)\in k^{N+1}.$$
Also, since $\bigcap_{j=0}^{l_i}Q_{I_i(j)}^*\cap V(\bar k)=\varnothing$, from the compactness of $V$ and the finiteness of $S$ there exist a positive constant $c_1$ such that
$$ \|P_{i,j}({\bf x})\|_v\le c_1\max_{0\le s\le t_{i,j}}\|Q_{I_i(s)}({\bf x})\|_v, $$
and $c_2$ such that
$$ \|{\bf x}\|_v^d\le c_2\max_{0\le j\le l_i}\|Q_{I_i(j)}({\bf x})\|_v$$
for all $v\in S$ and ${\bf x}\in V(k)$. 

For a given $v\in S$ and for a fixed point ${\bf x}\in V(k)$, there exists an index $i$ such that
$$\|Q_{I_i(1)}({\bf x})\|_v\le \|Q_{I_i(2)}({\bf x})\|_v\le\cdots\le \|Q_{I_1(q)}({\bf x})\|_v.$$

Therefore, we have
\begin{align*}
\prod_{i=1}^q\dfrac{\|{\bf x}\|_v^d}{\|Q_i({\bf x})\|_v}&\le c_2^{q-l_i}\prod_{j=0}^{l_i-1}\dfrac{\|{\bf x}\|_v^d}{\|Q_{I_i(j)}({\bf x})\|_v}\\
&\le c_2^{q-l_i}c^{-l_i}\prod_{j=0}^{n-1}\left(\dfrac{c\|{\bf x}\|_v^{d}}{\|Q_{I_i(i)}({\bf x})\|_v}\right )^{t_{i,j+1}-t_{i,j}}\\
&\le c_2^{q-l_i}c^{-l_i}\prod_{j=0}^{n-1}\left(\dfrac{c\|{\bf x}\|_v^{d}}{\|Q_{I_i(i)}({\bf x})\|_v}\right )^{\Delta}\\
&\le c_2^{q-l_i}c^{n\Delta -l_i}c_1^{n\Delta}\prod_{j=0}^{n-1}\left(\dfrac{\|{\bf x}\|_v^{d}}{\|P_{i,j}({\bf x})\|_v}\right )^{\Delta}\\
&\le C\prod_{j=0}^{n}\left(\dfrac{\|{\bf x}\|_v^{d}}{\|P_{i,j}({\bf x})\|_v}\right )^{\Delta},
\end{align*}
where $C$ is a positive constant, chosen independently from $v$ and ${\bf x}$.

The above inequality implies that
\begin{align}\label{3.2}
\log \prod_{i=1}^q\dfrac{\|{\bf x}\|_v^d}{\|Q_i({\bf x})\|_v}\le \Delta\log\sum_{i=1}^{n_0}\dfrac{\|{\bf x}\|_v^{(n+1)d}}{\prod_{j=0}^{n}\|P_{i,j}({\bf x})\|_v}+O(1),\ \forall {\bf x}\in V(k),
\end{align}
where $O(1)$ depending only on $Q_i\ (1\le i\le q)$.
We define 
$${\bf c}_{v,{\bf x}} = (c_{1,0},\ldots,c_{1,n},\ldots,c_{n_0,0},\ldots,c_{n_0,n})\in\mathbb R^{l},$$ 
where
\begin{align}\label{3.4}
c_{i,j}:=\log\frac{c\|{\bf x}\|_v^d}{\|P_{i,j}({\bf x})\|_v}\text{ for } i=1,\ldots ,n_0 \text{ and }j=0,\ldots ,n.
\end{align}
Choose ${\bf a}_{1},\ldots ,{\bf a}_{H_Y(u)}\in\mathbb N^{l}$ with
$$ {\bf a}_{s}=(a_{s,1,0},\ldots,a_{s,1,n},\ldots,a_{s,n_0,0},\ldots,a_{s,n_0,n}), a_{s,i,j}\in\{1,\ldots,\xi_u\}, $$
 such that the residue classes of ${\bf y}^{{\bf a}_{1}},\ldots ,{\bf y}^{{\bf a}_{H_Y(u)}}$ (modulo $(I_Y)_u$) form a basis of $Y_u$ and
\begin{align}\label{3.5}
S_Y(u,{\bf c}_{v,{\bf x}})=\sum_{i=1}^{H_Y(u)}{\bf a}_{i}\cdot{\bf c}_{v,{\bf x}}.
\end{align}
Then we may write
$$ {\bf y}^{{\bf a}_{s}}=L_{s,v,{\bf x}}(v_1,\ldots ,v_{H_Y(u)}), $$ 
where $L_{s,v,{\bf x}}\ (1\le s\le H_Y(u))$ are independent linear forms.
We have
\begin{align*}
\log\prod_{s=1}^{H_Y(u)} \|L_{s,v,{\bf x}}(F({\bf x}))\|_v&=\log\prod_{s=1}^{H_Y(u)}\prod_{\underset{0\le j\le n}{1\le i\le n_0}}\|P_{i,j}({\bf x})\|_v^{a_{s,i,j}}\\
&=-S_Y(m,{\bf c}_{v,{\bf x}})+duH_Y(u)\log \|{\bf x}\|_v .
\end{align*}
This implies that
\begin{align*}
\log\prod_{s=1}^{H_Y(u)}\dfrac{\|F({\bf x})\|_v\cdot \|L_{s,v,{\bf x}}\|_v}{\|L_{s,v,{\bf x}}(F({\bf x}))\|_v}=&S_Y(u,{\bf c}_{v,{\bf x}})-duH_Y(u)\log \|{\bf x}\|_v \\
&+H_Y(u)\log \|F({\bf x})\|_v+O(H_Y(u)).
\end{align*}
Note that $c_{i,j}$ depend on ${\bf x}$, so also the ${\bf a}_s$, but there are only finitely many possibilities for the ${\bf a}_s$ once $u$ is given. Therefore, $L_{s,v,{\bf x}}$ depends on $s$ and ${\bf x}$, but the number of these linear forms is finite. Denote by $\mathcal L_v$ the set of all $L_{s,v,{\bf x}}$ occurring in the above inequalities (when $s$ and ${\bf x}$ vary). Then we have
\begin{align}\label{3.6}
\begin{split}
S_Y(u,{\bf c}_{v,{\bf x}})\le&\max_{\mathcal J\subset\mathcal L_v}\log\prod_{L\in \mathcal J}\dfrac{\|F({\bf x})\|_v\cdot \|L\|_v}{\|L(F({\bf x}))\|_v}+duH_Y(u)\log \|{\bf x}\|_v\\
& -H_Y(u)\log \|F({\bf x})\|_v+O(H_Y(u)),
\end{split}
\end{align}
where the maximum is taken over all subsets $\mathcal J\subset\mathcal L_v$ with $\sharp\mathcal J=H_Y(u)$ and $\{L;L\in\mathcal J\}$ is linearly independent.
From Theorem \ref{2.12} we have
\begin{align}\label{3.7}
\dfrac{1}{uH_Y(u)}S_Y(u,{\bf c}_{v,{\bf x}})\ge&\frac{1}{(n+1)\delta}e_Y({\bf c}_{v,{\bf x}})-\frac{(2n+1)\delta}{u}\max_{\underset{0\le j\le n}{1\le i\le n_0}}c_{i,j}
\end{align}
It is clear that
\begin{align*}
\max_{\underset{0\le j\le n}{1\le i\le n_0}}c_{i,j}\le \sum_{0\le j\le n}\log\frac{c\|{\bf x}\|_v^d}{\|P_{1,j}({\bf x})\|_v}+O(1),
\end{align*}
where $O(1)$ does not depend on ${\bf x}$ and $v$. Combining (\ref{3.6}), (\ref{3.7}) and the above remark, we get
\begin{align}\nonumber
\frac{1}{(n+1)\delta}e_Y({\bf c}_{v,{\bf x}})\le &\dfrac{1}{uH_Y(u)}\left (\max_{\mathcal J\subset\mathcal L_v}\log\prod_{L\in \mathcal J}\dfrac{\|F({\bf x})\|_v\cdot \|L\|_v}{\|L(F({\bf x}))\|_v}-H_Y(u)\log \|F({\bf x})\|_v\right )\\
\label{3.8}
\begin{split}
&+d\log \|{\bf x}\|_v+\frac{(2n+1)\delta}{u}\max_{\underset{0\le j\le n}{1\le i\le n_0}}c_{i,j}+O(1/u)\\
\le &\dfrac{1}{uH_Y(u)}\left (\max_{\mathcal J\subset\mathcal L_v}\log \prod_{L\in\mathcal J}\dfrac{\|F({\bf x})\|_v\cdot \|L\|_v}{\|L(F({\bf x}))\|_v}-H_Y(u)\log \|F({\bf x})\|_v\right )\\
&+d\log \|{\bf x}\|_v+\frac{(2n+1)\delta}{u}\sum_{0\le j\le n}\log\frac{c\|{\bf x}\|_v^d}{\|P_{1,j}({\bf x})\|_v}+O(1/u).
\end{split}
\end{align}
Since $P_{1,0},\ldots,P_{1,n}$ are in general with respect to $V$, By Lemma \ref{2.13}, we have
\begin{align}\label{3.9}
e_Y({\bf c}_{v,{\bf x}})\ge (c_{1,0}+\cdots +c_{1,n})\cdot\delta =\left (\sum_{0\le j\le n}\log\frac{c\|{\bf x}\|_v^d}{\|P_{1,j}({\bf x})\|_v}\right )\cdot\delta.
\end{align}
Combining (\ref{3.8}) and (\ref{3.9}), we get
\begin{align*}
\sum_{0\le j\le n}\log\frac{c\|{\bf x}\|_v^d}{\|P_{1,j}({\bf x})\|_v}\le&\dfrac{n+1}{uH_Y(u)}\left (\max_{\mathcal J\subset\mathcal L_v}\log\prod_{L\in \mathcal J}\dfrac{\|F({\bf x})\|_v\cdot \|L\|_v}{\|L(F({\bf x}))\|_v}-H_Y(u)\log \|F({\bf x})\|_v\right )\\
&+d(n+1)\log \|{\bf x}\|_v+\frac{(n+1)(2n+1)\delta}{n}\max_{\underset{0\le j\le n}{1\le i\le n_0}}c_{i,j}+O(1/u)\\
\le &\dfrac{n+1}{uH_Y(u)}\left (\max_{\mathcal J\subset\mathcal L_v}\log \prod_{L\in\mathcal J}\dfrac{\|F({\bf x})\|_v\cdot \|L\|_v}{\|L(F({\bf x}))\|_v}-H_Y(u)\log \|F({\bf x})\|_v\right )\\
&+d(n+1)\log \|{\bf x}\|_v\\
&+\frac{(n+1)(2n+1)\delta}{u}\sum_{0\le j\le n}\log\frac{\|{\bf x}\|_v^d\cdot\|P_{1,j}\|_v}{\|P_{1,j}({\bf x})\|_v}+O(1/u).
\end{align*}
Then, from (\ref{3.2}) and the above inequality we have
\begin{align}\label{3.10}
\begin{split}
\frac{1}{\Delta}&\left (\log \prod_{i=1}^q\dfrac{\|{\bf x}\|_v^d}{\|Q_i({\bf x})\|_v}+O(1)\right )\\
&\le\dfrac{n+1}{uH_Y(u)}\left (\max_{\mathcal J\subset\mathcal L_v}\log\prod_{L\in\mathcal J}\dfrac{\|F({\bf x})\|_v\cdot \|L\|_v}{\|L(F({\bf x}))\|_v}-H_Y(u)\log \|F({\bf x})\|_v\right )\\
&+d(n+1)\log \|{\bf x}\|_v+\frac{(2n+1)(n+1)\delta}{u}\sum_{\underset{0\le j\le n}{1\le i\le n_0}}\log\frac{c\|{\bf x}\|_v^d}{\|P_{i,j}({\bf x})\|_v},
\end{split}
\end{align}
where $O(1)$ does not depend on ${\bf x}$. 

Summing-up both sides of the above inequalities over all $v\in S$ with noting that $S$ is a finite set, we obtain
\begin{align}\nonumber
\frac{1}{\Delta}\sum_{v\in S}\sum_{j=1}^q\lambda_{Q_j,v}(x)\le &\dfrac{n+1}{uH_Y(u)}\sum_{v\in S}\max_{\mathcal J\subset\mathcal L_v}\log\prod_{L\in\mathcal J}\dfrac{\|F({\bf x})\|_v\cdot \|L\|_v}{\|L(F({\bf x}))\|_v}\\
\label{3.11}
\begin{split}
&-\dfrac{n+1}{u}\sum_{v\in S}\log \|F({\bf x})\|_v+d(n+1)\sum_{v\in S}\log \|{\bf x}\|_v\\
&+\frac{(2n+1)(n+1)\delta}{u}\sum_{\underset{0\le j\le n}{1\le i\le n_0}}\sum_{v\in S}\log\frac{\|{\bf x}\|_v^d\cdot\|P_{i,j}\|_v}{\|P_{i,j}({\bf x})\|_v}+O(1).
\end{split}
\end{align} 
Note that the above inequality does not depend on the choice of components of ${\bf x}$. By enlarging $S$, if necessary, we may assume that $S$ contains $M^\infty_k$, and all coefficients of all polynomials $P_{i,j}$ are $S$-integers. Then, for each $v\in M_k\setminus S$, 
$$ \|P_{i,j}({\bf x})\|_v\le \|{\bf x}\|^d_v\cdot \|P_{i,j}\|_v\le \|{\bf x}\|^d_v \text{ for all }i,j,$$
and hence by the definition of $F$ we have
$$ \|F({\bf x})\|_v\le \|{\bf x}\|^{du}_v,\text{ i.e., }-\frac{1}{u}\log \|F({\bf x})\|_v+d\log \|{\bf x}\|_v\ge 0.$$
This implies that 
$$-\dfrac{1}{u}\sum_{v\in S}\log \|F({\bf x})\|_v+d\sum_{v\in S}\log \|{\bf x}\|_v\le -\dfrac{1}{u}\sum_{v\in M_k}\log \|F({\bf x})\|_v+d\sum_{v\in M_k}\log \|{\bf x}\|_v.$$
Then we have
\begin{align}\label{new}
-\dfrac{n+1}{u}\sum_{v\in S}\log \|F({\bf x})\|_v+d(n+1)\sum_{v\in S}\log \|{\bf x}\|_v\le -\frac{n+1}{u}h(F({\bf x}))+dh({\bf x}).
\end{align}
Since $\|P_{i,j}({\bf x})\|_v\le \|{\bf x}\|^d_v\cdot \|P_{i,j}\|_v$ for all $i,j$ and $v\in M_k\setminus S$, we have
\begin{align}\label{new2}
\sum_{v\in S}\log\frac{\|{\bf x}\|_v^d\cdot\|P_{i,j}\|_v}{\|P_{i,j}({\bf x})\|_v}\le \sum_{v\in M(k)}\log\frac{\|{\bf x}\|_v^d\cdot\|P_{i,j}\|_v}{\|P_{i,j}({\bf x})\|_v}=dh({\bf x})+h(P_{i,j}).
\end{align}
By the subspace theorem due to Schlickewei \cite{Sch77} (see also Schmidt \cite[Theorem 3]{Sch75}), for every $\epsilon' >0$ and all $F({\bf x})$ outside a finite union of proper linear subspaces, we have
\begin{align*}
\sum_{v\in S}\max_{\mathcal J\subset\mathcal L_v}\log\prod_{L\in\mathcal J}\dfrac{\|F({\bf x})\|_v\cdot \|L\|_v}{\|L(F({\bf x}))\|_v}\le (H_Y(u)+\epsilon') h(F({\bf x})).
\end{align*}
Combining (\ref{3.11}), (\ref{new}), (\ref{new2}) and this inequality, we have
\begin{align}\label{3.12}
\begin{split}
\frac{1}{\Delta}\sum_{v\in S}\sum_{j=1}^q\lambda_{Q_j,v}({\bf x})\le& \dfrac{\epsilon'(n+1)}{uH_Y(u)}h(F({\bf x}))\\
&+\left(d(n+1)+\frac{(2n+1)(n+1)l\delta}{u}\right)h({\bf x})+O(1)\\
\le&\left(d(n+1)+\frac{(2n+1)(n+1)l\delta}{u}+\dfrac{\epsilon'(n+1)d}{H_Y(u)}\right)h({\bf x})+O(1),
\end{split}
\end{align}
for all ${\bf x}\in V$ outside a finite union of proper algebraic subsets.

Note that $h(F({\bf x}))\le uh({\bf x})$. Then by choosing
$$ u\ge 4\Delta(2n+1)(n+1)l\delta \epsilon^{-1} \ \text{ and } \epsilon'\le \frac{H_Y(u)\epsilon}{4\Delta(n+1)d},$$
from (\ref{3.12}) we get
$$ \sum_{v\in S}\sum_{j=1}^q\dfrac{\lambda_{Q_j,v}({\bf x})}{d}\le  \left(\Delta(n+1)+\frac{\epsilon}{2}\right)h({\bf x})+C,$$
for all ${\bf x}\in V$ outside a finite union of proper algebraic subsets, where $C$ is a positive constant.
Since there only finite points ${\bf x}\in V$ such that $h({\bf x})$ is bounded above by $\dfrac{2C}{\epsilon}$, then
$$ \sum_{v\in S}\sum_{j=1}^q\dfrac{\lambda_{Q_j,v}({\bf x})}{d}\le  \left(\Delta(n+1)+\epsilon\right)h({\bf x})$$
for all ${\bf x}\in V$ outside a finite union of proper algebraic subsets. The theorem is proved.
\end{proof}

\vskip0.2cm
\noindent
{\bf Disclosure statement:} The author states that there is no conflict of interest. 
% 
% \vskip0.2cm
% \noindent
% {\bf Data Availability Statements:} The datasets generated during and/or analysed during the current study are available in the documents in the reference list.

\vskip0.2cm
\noindent
{\bf Funding:} This research is funded by Vietnam National Foundation for Science and Technology Development (NAFOSTED) under grant number 101.02-2021.12.

\vskip0.2cm
{\footnotesize 
\noindent
{\sc Si Duc Quang}\\
$^1$ Department of Mathematics, Hanoi National University of Education,\\
 136-Xuan Thuy, Cau Giay, Hanoi, Vietnam.\\
$^2$ Thang Long Institute of Mathematics and Applied Sciences,\\
 Nghiem Xuan Yem, Hoang Mai, HaNoi, Vietnam.\\
\textit{E-mail}: quangsd@hnue.edu.vn}

\end{document}